\documentclass[12pt]{amsart}
\usepackage[english]{babel}
\usepackage{amsmath}
\usepackage{amsthm}
\usepackage{mathtools}
\usepackage{mathrsfs}
\usepackage{amssymb}
\usepackage{xcolor}
\usepackage{xfrac}
\usepackage{esint}
\usepackage{graphicx}
\usepackage{float}
\usepackage{array}
\usepackage{enumitem}
\usepackage{tikz-cd}
\usepackage{centernot}
\usepackage{stmaryrd}
\usepackage{comment}
\usepackage{multirow}
\excludecomment{confidential}

\numberwithin{equation}{section}

\newtheorem{theorem}{Theorem}[section]
\newtheorem{corollary}[theorem]{Corollary}

\newtheorem{proposition}[theorem]{Proposition}
\newtheorem{definition}[theorem]{Definition}

\newtheorem{conjecture}[theorem]{Conjecture}

\DeclareMathOperator{\Alg}{Alg}
\DeclareMathOperator{\Ker}{Ker}
\DeclareMathOperator{\ord}{ord}

\DeclareMathOperator{\Spec}{Spec}
\DeclareMathOperator{\alg}{alg}

\DeclareMathOperator{\Mat}{Mat}
\DeclareMathOperator{\Vect}{Vect}
\DeclareMathOperator{\Hom}{Hom}

\DeclareMathOperator{\Tor}{Tor}
\DeclareMathOperator{\codim}{codim}
\DeclareMathOperator{\Aut}{Aut}

\newcommand{\field}[1] {\mathbb{#1}}

\newcommand{\Z}{\field{Z}}

\newcommand{\K}{\field{K}}

\def\e{\varepsilon}

\def\g{\gamma}

\def\l{\lambda}

\def\L{\Lambda}
\def\o{\omega}
\def\O{\Omega}

\def\s{\sigma}

\def\v{\varphi}

\newcommand{\mc}{\mathcal}
\newcommand{\mf}{\mathfrak}

\DeclareMathOperator{\sign}{sign}

\begin{document}
	\title{A bridge between Vector Bundle Theory and Nonlinear Spectral Theory}
	\author{Juli\'an L\'opez-G\'omez, Juan Carlos Sampedro} \thanks{The authors have been supported by the Research Grant PGC2018-097104-B-I00 of the Spanish Ministry of Science, Technology and Universities and by the Institute of Interdisciplinar Mathematics of Complutense University. The second author has been also supported by PhD Grant PRE2019\_1\_0220 of the Basque Country Government.}
\address{Institute of Interdisciplinary Mathematics \\
	Department of Mathematical Analysis and Applied Mathematics \\
	Complutense University of Madrid \\
	28040-Madrid \\
	Spain.}
\email{julian@mat.ucm.es, juancsam@ucm.es}


\begin{abstract}
	This paper establishes some hidden connections
 between the theory of generalized algebraic multiplicities, the intersection index of algebraic varieties, and the notion of orientability of vector bundles. The novel approach adopted in it facilitates the definition of several invariants closely related to the first Stiefel--Whitney fundamental class through some path integration techniques.
\end{abstract}

\maketitle

\section{Introduction}

\noindent In 1988, the classical notions of algebraic multiplicity of eigenvalues of linear operators was generalized to an infinite dimensional Fredholm setting by Esquinas and L\'{o}pez-G\'{o}mez \cite{ELG} to characterize the nonlinear eigenvalues of the $\mc{C}^r$-curves of linear Fredholm operators in the context of local bifurcation theory (see also Esquinas \cite{Es} and L\'{o}pez-G\'{o}mez \cite{LG01}). More recently, in  \cite{JJ3}, by adopting a geometrical point of view, the authors  established  a  hidden connection between the generalized algebraic multiplicity of Esquinas and L\'{o}pez-G\'{o}mez \cite{ELG,Es,LG01} and the
concept of local intersection index of algebraic varieties, a central device in algebraic geometry.  Precisely, it was shown that the algebraic multiplicity of a curve of Fredholm operators of index zero $\mf{L}:[0,1]\to\Phi_{0}(U)$ equals the intersection index of the curve $\mf{L}([0,1])\subset \Phi_{0}(U)$ with the subspace of singular operators $\mathcal{S}(U)\subset \Phi_{0}(U)$. Throughout this paper, $\Phi_{0}(U)$ stands for the space of bounded linear operators $T: U\to U$ on a real Banach space $U$ that are Fredholm of index zero. Naturally, this bridge between nonlinear spectral theory and algebraic geometry provides the generalized algebraic multiplicity with a deep geometrical meaning.
\par 
The main goal of this paper is studying vector bundles via topological $K$-Theory focusing special attention into the obstruction described by the first Stiefel--Whitney characteristic class, $\omega_{1}$. In particular, it will be shown that this obstruction can be completely described from the notion of generalized algebraic multiplicity, and hence, through the concept of intersection index of algebraic varieties, as discussed by the authors in \cite{JJ3}. These connections are  established in this paper by combining the Atiyah--J\"{a}nich index map
\begin{equation}
\label{E0}
\mathfrak{Ind}\;:\; [X,\Phi_{0}(U)]\longrightarrow\tilde{K}\mathcal{O}(X),
\end{equation}
with some spectral techniques developed by the authors in a series of recent papers \cite{JJ,JJ2,JJ3} for
continuous maps $h:X\to \Phi_{0}(U)$. In \eqref{E0}, $[X,\Phi_{0}(U)]$ stands for the set of homotopy classes of continuous maps $X\to \Phi_{0}(U)$, and $\tilde{K}\mathcal{O}(X)$ denotes the real reduced $K$-group of the compact path connected topological space $X$; the $K$-group consists of the stable equivalence classes of real vector bundles with base space $X$. In particular, we will be able to characterize the orientability of a given vector bundle in terms of the algebraic multiplicity and the local intersection index of certain algebraic varieties.  
\par
The relationship between the obstruction associated to the first Stiefel--Whitney class and the concept of intersection index can be described as follows. Under the appropriate assumptions, the index map \eqref{E0} is an isomorphism and hence, each real vector bundle $E\to X$ corresponds to a single parametrized family of Fredholm operators $h:X\to\Phi_{0}(U)$. Roughly spoken, we will establish that, considering $\omega_{1}(E)$ as a map $\pi_{1}(X)\to \mathbb{Z}_{2}$, the value $\omega_{1}(E)(\gamma)$ equals the sign of the algebraic multiplicity of the closed curve $$h\circ\gamma: \mathbb{S}^{1}\to\Phi_{0}(U).$$ Equivalently, $\omega_{1}(E)(\gamma)$ is the sign of the intersection index between $[h\circ\gamma](\mathbb{S}^{1})\subset \Phi_{0}(U)$ and the subspace of singular operators $\mathcal{S}(U)\subset \Phi_{0}(U)$. Therefore, the information provided by the first Stiefel--Whitney class can be translated into the way that a geometrical object in  $\Phi_{0}(U)$, as a parameterized family of Fredholm operators, intersects with $\mathcal{S}(U)$.
\par 
At a further stage, thanks to the flexibility and versatility of our new approach, by using path integration techniques on Riemannian manifolds, we will be able to define a new topological invariant of stable equivalence classes of real vector bundles via the integration of loops of the base space $X$. Such an invariant has been named in this paper as the \emph{global torsion invariant}. Based on the spectral techniques discussed previously, we will find out  the value of this invariant in some particular examples, and will be able to characterize the orientability of some real bundles in terms of it.
\par
The organization of this paper is the following. Section 2 reviews briefly the concept of
generalized algebraic multiplicity, $\chi$, introduced by Esquinas and L\'opez-G\'omez in \cite{ELG,Es,LG01}. Section 3 sketches the relationship between $\chi$ and the intersection index of algebraic varieties, $i$. Section 4 introduces the central concept of parity, $\sigma$, discussed by Fitzpatrick, Pejsachowicz and Rabier \cite{FPRa} and discusses its relations with $\chi$. The parity $\sigma$, which can be determined from $\chi$, establishes a bridge between the algebraic information provided by $\chi$ and some pivotal topological aspects of $\Phi_0(U)$, as it is a central invariant to study the topology of $\Phi_{0}(U)$. Section 5 describes the relationship between the first Stiefel--Whitney characteristic class, $\omega_1$, and the algebraic multiplicity, or, equivalently, the intersection index. It studies also the several notions of orientability of vector bundles and parameterized families of Fredholm operators. Section 6 introduces the global torsion invariant, a topological invariant of stable equivalence classes of vector bundles defined via $\chi$ through some well known path integration techniques. Finally, Section 7 finds out the global torsion invariant of the circle and the $n$-dimensional torus.

\section{The generalized algebraic multiplicity}

\noindent In this section we revise in a self contained way some fundamental properties of nonlinear spectral theory and, in particular, some fundamentals on the generalized algebraic multiplicity, $\chi$, introduced by Esquinas and L\'{o}pez-G\'{o}mez \cite{ELG} in 1988, and later developed in \cite{Es} and
\cite{LG01}. Throughout it, $\mathbb{K}\in\{\mathbb{R},\mathbb{C}\}$ and $\Omega$ stands for a subdomain of $\mathbb{K}$, and, for any given finite dimensional curve $\mathfrak{L}\in\mathcal{C}(\Omega,\mathcal{L}(\mathbb{K}^{N}))$,
a point $\l\in\Omega$ is said to be a \textit{generalized eigenvalue} of $\mathfrak{L}$ if $\mathfrak{L}(\l)\notin GL(\mathbb{K}^{N})$, i.e., $\mathrm{det\,}\mf{L}(\l)=0$. Then, the
\textit{generalized spectrum} of $\mathfrak{L}\in\mathcal{C}(\Omega,\mathcal{L}(\mathbb{K}^{N}))$
is defined by
\begin{equation*}
\Sigma(\mathfrak{L}):=\{\lambda\in\Omega: \mathfrak{L}(\lambda)\notin GL(\mathbb{K}^{N})\}.
\end{equation*}
For analytic curves $\mathfrak{L}\in\mathcal{H}(\O,\mathcal{L}(\mathbb{K}^{N}))$, since $\mathrm{det\,}\mf{L}(\l)$ is analytic in $\l\in\O$, either $\Sigma(\mathfrak{L})=\Omega$,
or $\Sigma(\mathfrak{L})$ is discrete. Thus, $\Sigma(\mf{L})$ consists of isolated generalized
eigenvalues if $\mf{L}(\mu)\in GL(\mathbb{K}^N)$ for some $\mu\in\O$. In such case, the \textit{algebraic multiplicity} of the curve $\mathfrak{L}\in\mathcal{H}(\Omega,\mathcal{L}(\mathbb{K}^{N}))$ at $\lambda_{0}$ is defined through
\begin{equation}
\label{2.1}
\mathfrak{m}_{\alg}[\mathfrak{L},\lambda_{0}]:=\ord_{\lambda_{0}}\det\mathfrak{L}(\lambda).
\end{equation}
This concept extends the classical notion of algebraic multiplicity in linear algebra. Indeed,
if $\mathfrak{L}(\lambda)=\lambda I_{N}-T$ for some linear operator $T\in\mathcal{L}(\mathbb{K}^{N})$, then $\mathfrak{L}\in\mathcal{H}(\mathbb{K},\mathcal{L}(\mathbb{K}^{N}))$ and it is easily seen that $\mathfrak{m}_{\alg}[\mathfrak{L},\lambda_{0}]$ is well defined for all $\l_0\in\Sigma(\mf{L})$ and that it actually equals the classical notion of multiplicity in linear algebra
\begin{equation}
\label{ii.2}
\mathfrak{m}_{\alg}[\mathfrak{L},\lambda_{0}]=\ord_{\lambda_{0}}\det(\lambda I_{N}-T).
\end{equation}
Indeed, since $GL(\K^N)$ is open, for sufficiently large $\l$ we have that $I_N-\l^{-1} T\in GL(\K^N)$. Thus, $\l I_N-T\in GL(\K^N)$ and $\Sigma(\mf{L})$ is discrete.
\par
This concept admits a natural (non-trivial) extension to an infinite-dimensional setting. To formalize it, we need to introduce some of notation. In this paper, for any given $\mathbb{K}$-Banach space, say $U$, we denote by $\Phi_0(U)$  the set of linear Fredholm operators of index zero in $U$. Then, a \emph{Fredholm path,} or curve, is any map $\mathfrak{L}\in \mathcal{C}(\Omega,\Phi_{0}(U))$.  Naturally, for any given $\mathfrak{L}\in \mathcal{C}(\Omega,\Phi_{0}(U))$, it is said that $\lambda\in\Omega$ is a \emph{generalized eigenvalue} of $\mathfrak{L}$ if $\mathfrak{L}(\lambda)\notin GL(U)$, and the \emph{generalized spectrum} of $\mathfrak{L}$, $\Sigma(\mathfrak{L})$,  is defined through   	 
\begin{equation*}
\Sigma(\mathfrak{L}):=\{\lambda\in\Omega: \mathfrak{L}(\lambda)\notin GL(U)\}.
\end{equation*}
The following concept, going back to \cite{LG01}, plays a pivotal role in the sequel.

\begin{definition}
	\label{de2.1}
	Let $\mathfrak{L}\in \mathcal{C}(\Omega, \Phi_{0}(U))$ and $\kappa\in\mathbb{N}$. A generalized eigenvalue $\lambda_{0}\in\Sigma(\mathfrak{L})$ is said to be $\kappa$-algebraic if there exists $\varepsilon>0$ such that
	\begin{enumerate}
		\item[{\rm (a)}] $\mathfrak{L}(\lambda)\in GL(U)$ if $0<|\lambda-\lambda_0|<\varepsilon$;
		\item[{\rm (b)}] There exists $C>0$ such that
		\begin{equation}
		\label{2.3}
		\|\mathfrak{L}^{-1}(\lambda)\|<\frac{C}{|\lambda-\lambda_{0}|^{\kappa}}\quad\hbox{if}\;\;
		0<|\lambda-\lambda_0|<\varepsilon;
		\end{equation}
		\item[{\rm (c)}] $\kappa$ is the minimal integer for which \eqref{2.3} holds.
	\end{enumerate}
\end{definition}
Throughout this paper, the set of $\kappa$-algebraic eigenvalues of $\mathfrak{L}$ will be  denoted by $\Alg_\kappa(\mathfrak{L})$, and the set of \emph{algebraic eigenvalues} by
\[
\Alg(\mathfrak{L}):=\bigcup_{\kappa\in\mathbb{N}}\Alg_\kappa(\mathfrak{L}).
\]
As in the special case when $U=\K^N$, according to Theorems 4.4.1 and 4.4.4 of \cite{LG01}, when $\mathfrak{L}(\lambda)$ is analytic in $\Omega$, i.e., $\mathfrak{L}\in\mathcal{H}(\Omega, \Phi_{0}(U))$,  then, either $\Sigma(\mathfrak{L})=\Omega$,
or $\Sigma(\mathfrak{L})$ is discrete and $\Sigma(\mathfrak{L})\subset \Alg(\mathfrak{L})$.
Subsequently, we denote by $\mathcal{A}_{\lambda_{0}}(\Omega,\Phi_{0}(U))$ the set  of curves $\mathfrak{L}\in\mathcal{C}^{r}(\Omega,\Phi_{0}(U))$ such that $\lambda_{0}\in\Alg_{\kappa}(\mathfrak{L})$ with $1\leq \kappa \leq r$ for some $r\in\mathbb{N}$.
Next, we will construct an infinite dimensional analogue of the classical algebraic multiplicity $\mathfrak{m}_{\alg}[\mathfrak{L},\l_{0}]$ for the class  $\mathcal{A}_{\lambda_{0}}(\Omega,\Phi_{0}(U))$. Essentially, this can be accomplished by
introducing an infinite-dimensional local concept of determinant in $\Phi_0(U)$, though  an infinite dimensional analogue of the classical notion of determinant is not available.
\par
Pick $T\in \Phi_0(U)$, and let  $P, Q\in \mathcal{L}(U)$, be projections onto $\Ker[T]$ and $R[T]$, respectively. Then, the pair $\mc{P}\equiv(P,Q)$ is refereed to as a pair of $T$-projections, and the following topological direct sum decompositions hold
\begin{equation}
\label{ii.4}
U =(I_{U}-P)(U)\oplus \Ker[T],\qquad U =R[T]\oplus(I_{U}-Q)(U).
\end{equation}
Moreover, setting
\begin{align*}
R[I_U-P]=(I_{U}-P)(U)\equiv \Ker[T]^{\perp}, \quad
R[I_U-Q] =(I_{U}-Q)(U)\equiv R[T]^{\perp},
\end{align*}
it follows from Fitzpatrick and Pejsachowicz \cite[p. 286]{FP1} that every $L\in \Phi_{0}(U)$ can be expressed as a block operator matrix
\begin{equation}
\label{ii.5}
L=\left(\begin{array}{cc} L_{11} & L_{12} \\[1ex] L_{21} & L_{22} \end{array}\right),
\end{equation}
where
\begin{equation*}
\begin{array}{ll}
L_{11}:=QL(I_{U}-P), & \quad L_{12}:=QLP, \\ L_{21}:=(I_{U}-Q)L(I_{U}-P), & \quad L_{22}:=(I_{U}-Q)LP.
\end{array}
\end{equation*}
In particular, since $TP=0$ and $(I_U-Q)T=0$, the operator $T$ can be expressed as
\begin{equation*}
T=\left(\begin{array}{cc}
T_{11} & 0 \\[1ex]
0 & 0
\end{array}\right)
\end{equation*}
with $T_{11}\in GL(\Ker[T]^{\perp}, R[T])$. Since $GL(\Ker[T]^{\perp},R[T])$ is open in $\mathcal{L}(\Ker[T]^{\perp},R[T])$ and $\Phi_{0}(U)$ is open in $\mathcal{L}(U)$,  there exists $\varepsilon>0$ such that, whenever $L\in \mathcal{L}(U)$ satisfies  $$\|T-L\|<\varepsilon,$$ then
$L\in\Phi_{0}(U)$, and it can be expressed as \eqref{ii.5} with  $L_{11}\in GL(\Ker[T]^{\perp},R[T])$. In this context, the \textit{Schur operator} of $T$ associated to the projection pair $\mc{P}\equiv(P,Q)$  can be defined through
\begin{equation*}
 \mathscr{S}_{T,\mc{P}}:  B_\varepsilon(T)\subset\Phi_{0}(U) \longrightarrow \mathcal{L}(\Ker[T],R[T]^{\perp}), \quad L \mapsto L_{22}-L_{21}L_{11}^{-1}L_{12}
\end{equation*}
where $B_\varepsilon(T)$ denotes the open ball of radius $\varepsilon$ centered at $T\in \Phi_0(U)$ in $\mathcal{L}(U)$. This operator, introduced by the authors in \cite{JJ3}, extends the  classical concept of \textit{Schur complement} of a matrix in the Euclidean space. Precisely, given any block matrix
\begin{equation}
\label{ii.6}
M=\left(\begin{array}{cc} A & B \\[1ex] C & D \end{array}\right),
\end{equation}
with $A\in GL(\mathbb{K}^{n})$, $B\in\Mat_{n\times m} (\mathbb{K})$, $C\in\Mat_{m\times n}(\mathbb{K})$ and $D\in \Mat_{m}(\mathbb{K})$, the \textit{Schur complement} of $D$ in $M$  is the matrix $M/A\in\Mat_{m}(\mathbb{K})$ defined by
\begin{equation*}
M/ A:=D-CA^{-1}B.
\end{equation*}
By  a lemma of Banachiewicz \cite[p. 50]{B}, the Schur complement satisfies the identity
\begin{equation*}
\det(M)=\det(A)\cdot\det(M/ A)
\end{equation*}
for every block matrix \eqref{ii.6}. Based on this feature, the authors introduced in \cite[Def. 3.1]{JJ3} the following local notion of determinant: Given $T\in\Phi_{0}(U)$ and any pair of $T$-projections $\mc{P}\equiv(P,Q)$, for sufficiently small $\varepsilon>0$, a local determinant functional can be defined through
\begin{equation*}
\mathfrak{D}_{T,\mc{P}}: B_{\varepsilon}(T)\subset \Phi_{0}(U)  \longrightarrow  \mathbb{K}, \quad
   L \mapsto \det(\mathscr{S}_{T,\mc{P}}(L)).
\end{equation*}
This notion of local determinant indeed behaves as a (local) determinant on $B_\varepsilon(T)$ since, for sufficiently small $\varepsilon>0$ and every $L\in B_\varepsilon(T)$,
$$
   L\in GL(U)\quad\hbox{if, and only if,} \quad \mathfrak{D}_{T,\mc{P}}(L)\neq0
$$
(see \cite[Th. 3.2]{JJ3} for a proof of this feature). Thanks to this (local) notion of determinant, we can introduce a generalized concept of algebraic multiplicity that can also be expressed in the vein of \eqref{2.1} even in an infinite-dimensional setting.

\begin{definition}
	\label{de2.2}
	Assume  $\mathfrak{L}\in\mathcal{A}_{\lambda_{0}}(\Omega,\Phi_{0}(U))$, i.e.,
	$\mathfrak{L}\in\mathcal{C}^{r}(\Omega,\Phi_{0}(U))$ with  $\lambda_{0}\in\Alg_{\kappa}(\mathfrak{L})$ for some integer $r\geq 1$ and $1\leq \kappa \leq r$. Then,
	the algebraic multiplicity of $\mathfrak{L}$ at $\lambda_{0}$ can be defined through
	\begin{equation}
	\label{ii.7}
	\mathfrak{m}_{\alg}[\mathfrak{L},\lambda_{0}]:= \ord_{\lambda_{0}}
	\mathfrak{D}_{\mathfrak{L}(\lambda_{0}),\mc{P}}(\mathfrak{L}(\lambda))
	\end{equation}
	for every pair $\mc{P}\equiv(P,Q)$ of $\mathfrak{L}(\lambda_{0})$-projections.
\end{definition}

According to \cite[Th. 1.1]{JJ3}, the formula \eqref{ii.7} can be equivalently expressed as
\begin{equation}
\label{ii.8}
\mathfrak{m}_{\alg}[\mathfrak{L},\lambda_{0}]=\ord_{\lambda_{0}}\det[P\mathfrak{L}^{-1}(\lambda)(I_{U}-Q)]^{-1}.
\end{equation}
By \cite[Th. 1.2]{JJ3}, this concept of multiplicity is consistent, in the sense that it does not depend on the particular choice of the pair $\mc{P}\equiv(P,Q)$ of $\mathfrak{L}(\lambda_{0})$-projections.
\par
Another approach to $\mathfrak{m}_{\alg}$,  useful for practical purposes, can be conducted through the theory of Esquinas and L\'{o}pez-G\'{o}mez
\cite{ELG},  where the following pivotal concept, generalizing the transversality condition of
Crandall and Rabinowitz \cite{CR},  was introduced. Subsequently, we will set
$\mathfrak{L}_{j}:=\frac{1}{j!}\mathfrak{L}^{(j)}(\lambda_{0})$, $1\leq j\leq r$, should these derivatives exist.

\begin{definition}
	\label{de2.3}
	Let $\mathfrak{L}\in \mathcal{C}^{r}(\O,\Phi_{0}(U))$ and $1\leq \kappa \leq r$. Then, a given $\lambda_{0}\in \Sigma(\mathfrak{L})$ is said to be a $\kappa$-transversal eigenvalue of $\mathfrak{L}$ if
	\begin{equation*}
	\bigoplus_{j=1}^{\kappa}\mathfrak{L}_{j}\left(\bigcap_{i=0}^{j-1}\Ker(\mathfrak{L}_{i})\right)
	\oplus R(\mathfrak{L}_{0})=U\;\; \hbox{with}\;\; \mathfrak{L}_{\kappa}\left(\bigcap_{i=0}^{\kappa-1}\Ker(\mathfrak{L}_{i})\right)\neq \{0\}.
	\end{equation*}
\end{definition}

For these eigenvalues, the following generalized concept of algebraic multiplicity was introduced by
Esquinas and L\'{o}pez-G\'{o}mez \cite{ELG},
\begin{equation}
\label{ii.9}
\chi[\mathfrak{L}, \lambda_{0}] :=\sum_{j=1}^{\kappa}j\cdot \dim \mathfrak{L}_{j}\left(\bigcap_{i=0}^{j-1}\Ker(\mathfrak{L}_{i})\right).
\end{equation}
In particular, when $\Ker[\mf{L}_0]=\mathrm{span}[\v_0]$ for some $\v_0\in U$ such that $\mf{L}_1\v_0\notin R[\mf{L}_0]$, then
\begin{equation}
\label{ii.10}
\mf{L}_1(\Ker[\mf{L}_0])\oplus R[\mf{L}_0]=U
\end{equation}
and hence, $\l_0$ is a 1-transversal eigenvalue of $\mf{L}(\l)$ with $\chi[\mf{L},\l_0]=1$. The transversality condition \eqref{ii.10} goes back to Crandall and Rabinowitz \cite{CR}. More generally, under condition \eqref{ii.10},
$$
\chi[\mf{L},\l_0]=\dim \Ker[\mf{L}_0].
$$
According to Theorems 4.3.2 and 5.3.3 of \cite{LG01}, for every $\mathfrak{L}\in \mathcal{C}^{r}(\O, \Phi_{0}(U))$, $\kappa\in\{1,2,...,r\}$ and $\lambda_{0}\in \Alg_{\kappa}(\mathfrak{L})$, there exists a polynomial $\Phi: \O\to \mathcal{L}(U)$ with $\Phi(\lambda_{0})=I_{U}$ such that $\lambda_{0}$ is a $\kappa$-transversal eigenvalue of the path
\begin{equation}
\label{ii.11}
\mathfrak{L}^{\Phi}:=\mathfrak{L}\circ\Phi\in \mathcal{C}^{r}(\O, \Phi_{0}(U)),
\end{equation}
and $\chi[\mathfrak{L}^{\Phi},\lambda_{0}]$ is independent of the curve of \emph{trasversalizing local isomorphisms} $\Phi$ chosen to transversalize $\mathfrak{L}$ at $\lambda_0$ through \eqref{ii.11}. Therefore, the following concept of multiplicity
is consistent
\begin{equation}
\label{ii.12}
   \chi[\mf{L},\l_0]:= \chi[\mathfrak{L}^{\Phi},\lambda_{0}].
\end{equation}
By a recent result of the authors, \cite[Th. 1.2]{JJ3},
\begin{equation}
\label{ii.13}
   \mathfrak{m}_{\alg}[\mf{L},\l_0]=\chi[\mf{L},\l_0]
\end{equation}
for all $\mathfrak{L}\in \mathcal{C}^{r}(\O, \Phi_{0}(U))$, $\kappa\in\{1,2,...,r\}$ and $\lambda_{0}\in \Alg_{\kappa}(\mathfrak{L})$. Any of these concepts of algebraic multiplicities, with its several equivalent formulations, can be easily extended by setting
$\chi[\mathfrak{L},\lambda_0] =0$ if $\lambda_0\notin\Sigma(\mathfrak{L})$ and
$\chi[\mathfrak{L},\lambda_0] =+\infty$ if $\lambda_0\in \Sigma(\mathfrak{L})
\setminus \Alg(\mathfrak{L})$ and $r=+\infty$. Thus, $\chi[\mathfrak{L},\lambda]$, or, equivalently,
$\mathfrak{m}_{\alg}[\mf{L},\l]$,  is well defined for all  $\lambda\in \O$ of any smooth path $\mathfrak{L}\in \mathcal{C}^{\infty}(\O,\Phi_{0}(U))$. In particular, for any analytical curve  $\mathfrak{L}\in\mathcal{H}(\O,\Phi_{0}(U))$.
The next uniqueness result, going back to Mora-Corral \cite{MC}, axiomatizes these concepts of algebraic multiplicity. Some refinements of them were delivered in \cite[Ch. 6]{LGMC}.

\begin{theorem}
	\label{th24}
	Let $U$ be a $\mathbb{K}$-Banach space. For every $\lambda_{0}\in\mathbb{K}$ and any open neighborhood $\Omega_{\lambda_{0}}\subset\mathbb{K}$ of $\lambda_{0}$, the algebraic multiplicity $\chi$ is the unique map 		
	\begin{equation*}
		\chi[\cdot, \lambda_{0}]: \mathcal{C}^{\infty}(\Omega_{\lambda_{0}}, \Phi_{0}(U))\longrightarrow [0,\infty]
	\end{equation*}
	such that
	\begin{enumerate}
		\item[{\rm (PF)}] For every pair $\mathfrak{L}, \mathfrak{M} \in \mathcal{C}^{\infty}(\Omega_{\lambda_{0}}, \Phi_{0}(U))$,
		\begin{equation*}
			\chi[\mathfrak{L}\circ\mathfrak{M}, \lambda_{0}]=\chi[\mathfrak{L},\lambda_{0}]+\chi[\mathfrak{M},\lambda_{0}].
		\end{equation*}
		\item[{\rm (NP)}] There exists a rank one projection $\Pi \in \mathcal{L}(U)$ such that
		\begin{equation*}
			\chi[(\lambda-\lambda_{0})\Pi +I_{U}-\Pi,\lambda_{0}]=1.
		\end{equation*}
	\end{enumerate}
\end{theorem}

The axiom (PF) is the  \emph{product formula} and (NP) is a \emph{normalization property}
for establishing the uniqueness of $\chi$. From these two axioms one can derive the remaining properties of  $\chi$; among them, that it equals the classical algebraic multiplicity when
\begin{equation*}
\mathfrak{L}(\lambda)= \lambda I_{U} - K
\end{equation*}
for some compact operator $K$. Indeed, for every $\mathfrak{L}\in \mathcal{C}^{\infty}(\Omega_{\lambda_{0}},\Phi_{0}(U))$, the following properties are satisfied:
\begin{itemize}
	\item $\chi[\mathfrak{L},\lambda_{0}]\in\mathbb{N}\uplus\{+\infty\}$;
	\item $\chi[\mathfrak{L},\lambda_{0}]=0$ if and only if $\mathfrak{L}(\lambda_0)
	\in GL(U)$;
	\item $\chi[\mathfrak{L},\lambda_{0}]<\infty$ if and only if $\lambda_0 \in\Alg(\mathfrak{L})$.
	\item If $U =\mathbb{K}^N$, then, in any basis,
	\begin{equation*}
	\label{1.1.18}
	\chi[\mathfrak{L},\lambda_{0}]= \mathrm{ord}_{\lambda_{0}}\det \mathfrak{L}(\lambda).
	\end{equation*}
	\item Let $K:U\to U$ be a compact operator, then, for every $\lambda_0\in \Spec(K)$,
	\begin{equation*}
	\label{1.1.90}
	\chi [\lambda I_U-K,\lambda_{0}]=\mathfrak{m}[K,\lambda_0],
	\end{equation*}
	where $\mathfrak{m}[K,\mu]$ is the classical algebraic multiplicity of $\mu$, i.e.,
	\[
	\mathfrak{m}[K,\mu]= \mathrm{dim\,}\mathrm{Ker}[(\mu I_{U}-K)^{\nu(\mu)}],
	\]
	where $\nu(\mu)$ is the \emph{algebraic ascent} of $\mu$, i.e. the minimal integer, $\nu\geq 1$, such that
	\[
	\mathrm{Ker}[(\mu I_{U}-K)^{\nu}]=\mathrm{Ker}[(\mu I_{U}-K)^{\nu+1}].
	\]
\end{itemize}

\section{The algebraic multiplicity and the intersection index}

\noindent As already noted, any generalized eigenvalue $\lambda_{0}\in\Sigma(\mathfrak{L})$ of the curve $\mathfrak{L}\in\mathcal{C}(\Omega,\Phi_{0}(U))$ lies in the intersection of $\mathfrak{L}$ with the space of singular operators $\mathcal{S}(U):=\Phi_{0}(U)\backslash GL(U)$. This geometrical feature leads the authors, \cite{JJ3}, to establish that the algebraic multiplicity actually measures how the curve $\mathfrak{L}(\lambda)$ intersects geometrically with $\mathcal{S}(U)$. Thus, relating the algebraic multiplicity with the \textit{local intersection index}, a pivotal geometrical device for measuring the nature of the intersections of algebraic varieties.
\par
Given a family of algebraic varieties, $\mathscr{V}$, an intersection theory over $\mathscr{V}$ consists  in giving a pairing
\begin{equation}
\label{1.10}
\bullet\,:\; A^{r}(X)\times A^{s}(X)\to A^{r+s}(X)
\end{equation}
satisfying a series of axioms (see, e.g.,  Hartshorne \cite[pp. 426--427]{H}, Eisenbud and Harris \cite[Ch. 1 and 2]{E}, and Fulton \cite[Ch. 7 and 8]{F}) for every $r,s \in \mathbb{N}$ and $X\in\mathscr{V}$, where $A^{r}(X)$ stands for the group of cycles of codimension $r$ on $X$ modulo rational equivalence. The graded group
$$
A(X)\equiv \bigoplus_{r\in\mathbb{N}}A^{r}(X)
$$
is referred to as the \textit{Chow group} after Chow \cite{Ch}. The pairing \eqref{1.10} gives to  $A(X)$ the structure of a graded ring, the \textit{Chow ring} of $X$. Giving an intersection theory to an algebraic variety $X$  consists in giving the structure of the Chow ring $A(X)$; the axioms of the pairing $\bullet$ trying to mimic in $X$ the celebrated \textit{Bezout theorem} \cite[\S 2.1.1]{E}. This explains why one of these axioms establishes that if $X_1$ and $X_2$ are subvarieties of $X$
with  \emph{proper intersection}, in the sense that any irreducible component of $X_1\cap X_2$ has codimension
$$
  \codim X_1 + \codim X_2,
$$
then
$$
[X_1]\bullet [X_2]=\sum_{j} i(X_1,X_2;C_{j})[C_{j}]
$$
where the sum runs over the set of all irreducible components, $C_{j}$, of $X_1\cap X_2$, and  the integer $i(X_1,X_2;C_{j})$ stands for the \textit{local intersection index} of $X_1$ and $X_2$ along $C_{j}$. By a result of Serre \cite[Ch. V, \S C.1]{S}, for any given pair $X_1, X_2$ of subvarieties of a smooth variety $X$ with proper intersection and every irreducible component, $C$,  of $X_1\cap X_2$, the local intersection index of $X_1$ and $X_2$ along $C$ is given through
$$
i(X_1,X_2;C)=\sum_{i=0}^{\dim X}(-1)^{i}\ell_{\mathcal{O}_{c,X}} \Tor^{\mathcal{O}_{c,X}}_{i}(\mathcal{O}_{c,X}/\mathfrak{P}_{X_1},\mathcal{O}_{c,X}/\mathfrak{P}_{X_2})
$$
where $\mathcal{O}_{c,X}$ is the local ring of $c\in C$ in $X$, and $\mathfrak{P}_{X_1}$ and $\mathfrak{P}_{X_2}$ are the ideals  of $X_1$ and $X_2$, respectively,  in the ring $\mathcal{O}_{c,X}$. In the Serre formula, we denote by $\ell_{R}(M)$ the length of the module $M$ over the ring $R$, and by $\Tor_{i}^{R}$  the $i$-th Tor $R$-module.
\par
Subsequently, $\mathcal{D}_{N}: \mathcal{L}(\mathbb{C}^{N})\to\mathbb{C}$ stands for the determinant map defined by $\mathcal{D}_{N}(T):=\det T$ for all $T\in \mathcal{L}(\mathbb{C}^{N})$. The next result of \cite{JJ3}, shows the exact relation between the local intersection index and the algebraic multiplicity. It establishes that  $\chi[\mathfrak{L},\l_{0}]$ is the local intersection index of the curve $\mathfrak{L}(\l)$ and the variety $\mathcal{D}_{N}^{-1}(0)$.

\begin{theorem}
	\label{th31}
	Let $T\in\mathcal{L}(\mathbb{C}^{N})$, $\lambda_{0}\in\sigma(T)$, and $\mathfrak{L}(\lambda)=\lambda I_{N}-T$, $\lambda\in \mathbb{C}$. Then,  	
	\begin{equation}
	\label{1.12}
	\mf{m}_{\alg}[\mathfrak{L},\lambda_{0}]=
	i(\mathcal{D}_{N}^{-1}(0),\mathfrak{L}(\mathbb{C});\mathfrak{L}(\lambda_{0})).
	\end{equation}
\end{theorem}
\noindent Under the assumptions of Theorem \ref{th31},
\begin{equation*}
\sum_{\lambda\in\sigma(T)}\mf{m}_{\alg}[\mathfrak{L},\lambda]=N.
\end{equation*}
Thus, by \eqref{1.12},
\begin{equation*}
\sum_{\lambda\in\sigma(T)}i(\mathcal{D}_{N}^{-1}(0),\mathfrak{L}(\mathbb{C});\mathfrak{L}(\lambda_{0}))
=\deg(\mathfrak{L})\deg(\mathcal{D}_{N}),
\end{equation*}
which provides us with an analogue of the Bezout theorem in this context.
\par
The relationship between the algebraic multiplicity and the intersection index can be extended to an infinite dimensional context through the concept of \emph{global linealization}. Given a complex Banach space $U$ and a domain $\O\subset \mathbb{C}$, for any given $\mathfrak{L}\in \mathcal{ A}_{\lambda_{0}}(\Omega,\Phi_{0}(U))$,  we will denote by
$$
\mathscr{L}_{\l_{0}}= \mathscr{L}_{\l_{0}}(\mathfrak{L}) \in \mathcal{L}(\mathbb{C}^M)
$$
the \emph{global linealization}, as discussed by Lemma 10.1.1
of L\'{o}pez-G\'{o}mez and Mora-Corral \cite{LGMC}, of the \emph{local Smith form} of the \textit{Schur complement} of $\mathfrak{L}$ at $\l_{0}$, whose existence follows from Theorem 7.4.1 of   \cite{LGMC}. To simplify the notation as much as possible, we will just refer to $\mathscr{L}_{\l_{0}}$ by the \textit{Schur reduction} of the curve $\mathfrak{L}$ at $\l_{0}$. The following result is a substantial generalization of Theorem \ref{th31}.

\begin{theorem}
\label{th32} For every $\mathfrak{L}\in \mathcal{A}_{\l_{0}}(\Omega,\Phi_{0}(U))$,
	\begin{equation}
	\label{1.13}
	\chi[\mathfrak{L},\lambda_{0}]=i(\mathcal{D}_{M}^{-1}(0),\lambda I_{M} - \mathscr{L}_{\l_{0}};\lambda_{0} I_{M}-\mathscr{L}_{\l_{0}}).
	\end{equation}
\end{theorem}

Although these bi-associations were stated in the field $\mathbb{C}$, by restricting ourselves
to analytical curves in $\mathbb{R}$ they are still valid in the real setting. Indeed, for every  $T\in\Phi_{0}(U)$, let us denote  by $T_{\mathbb{C}}\in \Phi_{0}(U_{\mathbb{C}})$ its complexification, i.e., its unique linear extension to $U_{\mathbb{C}}\to U_{\mathbb{C}}$, where
$$
 U_{\mathbb{C}}:=U \otimes_{\mathbb{R}} \mathbb{C}.
$$
Similarly, given any real Banach space $U$ and an analytic curve $\mathfrak{L}\in\mathcal{H}(I,\Phi_{0}(U))$, $\mathfrak{L}\equiv\mathfrak{L}(\l)$, $\l\in I$, where $I\subset\mathbb{R}$ is an interval, we can consider its complexification $\mathfrak{L}_{\mathbb{C}}\in\mathcal{H}(I_{\mathbb{C}},\Phi_{0}(U_{\mathbb{C}}))$ to be the analytic continuation of $[\mathfrak{L}(\cdot)]_{\mathbb{C}}$ to $I_{\mathbb{C}}:=I+i \mathbb{R}$.
Then, it is straightforward to verify that $\mathfrak{D}_{\mathfrak{L}_{\mathbb{C}}(\l_{0}),\mc{P}_{\mathbb{C}}}$, where $\mc{P}_{\mathbb{C}}\equiv(P_{\mathbb{C}},Q_{\mathbb{C}})$, is the complexification of $\mathfrak{D}_{\mathfrak{L}(\l_{0}),\mc{P}}$ for every $\l_{0}\in I$ and each pair $\mc{P}\equiv(P,Q)$ of $\mathfrak{L}(\lambda_{0})$-projections. Thus,
\begin{equation*}
\ord_{\lambda_{0}}\mathfrak{D}_{\mathfrak{L}(\l_{0}),\mc{P}}=\ord_{\lambda_{0}}\mathfrak{D}_{\mathfrak{L}_{\mathbb{C}}(\l_{0}),\mc{P}_{\mathbb{C}}}
\end{equation*}
and hence
\begin{equation*}
\chi[\mathfrak{L},\l_{0}]=\chi[\mathfrak{L}_{\mathbb{C}},\l_{0}]=i(\mathcal{D}_{M}^{-1}(0),\lambda I_{M} - \mathscr{L}_{\l_{0}}(\mathfrak{L}_{\mathbb{C}});\lambda_{0} I_{M}-\mathscr{L}_{\l_{0}}(\mathfrak{L}_{\mathbb{C}})),
\end{equation*}
where $\mathscr{L}_{\l_{0}}(\mathfrak{L}_{\mathbb{C}})\in\mathcal{L}(\mathbb{C}^{M})$ is the Schur reduction of the complexification $\mathfrak{L}_{\mathbb{C}}$ of $\mathfrak{L}$ at $\l_{0}$.

\section{The relationship between the parity and the algebraic multiplicity}

\noindent In this section, for a real Banach space $U$, we study the topology of the space of Fredholm operators of index zero,  $\Phi_{0}(U)$,  via the \textit{parity}, a topological invariant of paths introduced by Fitzpatrick, Pejsachowicz and Rabier \cite{FPRa}. This notion was introduced to overcome the lack of orientation in the infinite dimensional setting, allowing in this way, the definition of a topological degree for Fredholm operators of index zero (see \cite{FPRb}). Moreover, this section shows that the notions of parity and generalized algebraic multiplicity are strongly related.
\par
We begin by recalling some important features concerning the structure of $\Phi_{0}(U)$, which is an open subset (in general,  not linear) of $\mathcal{L}(U)$. Subsequently, for every $n\in\mathbb{N}$, we denote by $\mathcal{S}_{n}(U)$ the set of \textit{singular operators of order} $n$
\begin{equation*}
\mathcal{S}_{n}(U):=\{L\in\Phi_{0}(U):\;\; \dim \Ker[L] =n\}.
\end{equation*}
Then, the class of \textit{singular operators} is given through
\begin{equation*}
\mathcal{S}(U):=\Phi_{0}(U)\backslash GL(U)=\biguplus_{n\in\mathbb{N}}\mathcal{S}_{n}(U).
\end{equation*}
According to Fitzpatrick and Pejsachowicz \cite{FP1}, for every $n\in\mathbb{N}$, $\mathcal{S}_{n}(U)$ is a Banach submanifold of $\Phi_{0}(U)$ of codimension $n^{2}$. This allows us to view  $\mathcal{S}(U)$ as a stratified analytic set of $\Phi_{0}(U)$.
By Theorem 2 of Kuiper \cite{K}, the space of isomorphisms,
$GL({H})$, of any real or complex separable infinite dimensional Hilbert space, ${H}$,  is contractible and hence path-connected. Thus, in general, it is not possible to introduce an orientation for operators in $GL(U)$, since $GL(U)$ can be path-connected; by an orientation we mean the choice of a path connected component of the space $GL(U)$ when it contains al least two. This fact reveals a fundamental difference between finite and infinite dimensional spaces, as, for every $N\in\mathbb{N}$, it is folklore that $ GL(\mathbb{R}^{N})$ consists of two path-connected components, $GL^\pm(\mathbb{R}^N)$. A key technical tool to overcome this difficulty is provided by the concept of \emph{parity} introduced by  Fitzpatrick and Pejsachowicz \cite{FP2}.  The parity is a generalized local detector of the change of orientability of a given \emph{admissible path}. Although one cannot expect to get a global orientation in $\Phi_{0}(U)$ when $GL(U)$ is path-connected, one can study the orientability as a local phenomenon through the concept of \textit{parity}.
\par
Subsequently, a Fredholm path $\mathfrak{L}\in \mathcal{C}([a,b],\Phi_{0}(U))$ is said to be \emph{admissible} if $\mathfrak{L}(a), \mathfrak{L}(b)\in GL(U)$, and $\mathscr{C}([a,b],\Phi_{0}(U))$
stands for the set of those admissible paths. Moreover, for every $r\in\mathbb{N}\uplus\{+\infty\}$,
we set
\[
\mathscr{C}^r([a,b],\Phi_{0}(U)) := \mathcal{C}^{r}([a,b],\Phi_{0}(U))\cap \mathscr{C}([a,b],\Phi_{0}(U)),
\]
\[
\mathscr{H}([a,b],\Phi_{0}(U)) := \mathcal{H}([a,b],\Phi_{0}(U))\cap \mathscr{C}([a,b],\Phi_{0}(U)).
\]
The most geometrical way to introduce the notion of parity consists in defining it for $\mathscr{C}$-transversal  paths, and then for general admissible curves  through the density of $\mathscr{C}$-transversal paths in $\mathscr{C}([a,b],\Phi_{0}(U))$, established by Fitzpatrick and Pejsachowicz in \cite{FP1}. A  continuous Fredholm path, $\mathfrak{L}\in \mathcal{C}([a,b],\Phi_{0}(U))$, is said to be  \emph{$\mathscr{C}$-transversal} if
\begin{enumerate}
	\item[{\rm i)}] $\mathfrak{L}\in \mathscr{C}^{1}([a,b],\Phi_{0}(U))$;
	\item[{\rm ii)}] $\mathfrak{L}([a,b])\cap \mathcal{S}(U)\subset \mathcal{S}_{1}(U)$ and it is finite;
	\item[{\rm iii)}] $\mathfrak{L}$ is transversal to $\mathcal{S}_{1}(U)$ at each point of $\mathfrak{L}([a,b])\cap \mathcal{S}(U)$.
\end{enumerate}
When $\mathfrak{L}$ is $\mathscr{C}$-transversal, then, the (total) \emph{parity} of $\mathfrak{L}$ in $[a,b]$ is defined by
\begin{equation*}
\sigma(\mathfrak{L},[a,b]):=(-1)^{\kappa},
\end{equation*}
where $\kappa\in\mathbb{N}$ equals the cardinal of $\mathfrak{L}([a,b])\cap \mathcal{S}(U)$.
Thus, the parity of a $\mathscr{C}$-transversal path, $\mathfrak{L}(\lambda)$, is the number of times, mod 2, that $\mathfrak{L}(\lambda)$ intersects transversally the stratified analytic set $\mathcal{S}(U)$.
\par
The fact that the set of $\mathscr{C}$-transversal paths is dense in the set of admissible paths, $\mathscr{C}([a,b],\Phi_{0}(U))$, allows us
to define the parity for a general $\mathfrak{L}\in \mathscr{C}([a,b],\Phi_{0}(U))$ through
\begin{equation*}
\sigma(\mathfrak{L},[a,b]):=\sigma(\tilde{\mathfrak{L}},[a,b]),
\end{equation*}
where $\tilde{\mathfrak{L}}$ is any $\mathscr{C}$-transversal curve satisfying $\|\mathfrak{L}-\tilde{\mathfrak{L}}\|_{\infty}<\varepsilon$ for sufficiently small $\varepsilon>0$.
\par
The parity can be also introduced analytically as follows. Given $\mathfrak{L}\in\mathcal{C}([a,b],\Phi_{0}(U))$, a \textit{parametrix} of $\mathfrak{L}$ is a path $\mathfrak{P}\in \mathcal{C}([a,b],GL(U))$ satisfying
$$
   \mathfrak{P}(\l)\circ \mathfrak{L}(\l)-I_{U}\in\mathcal{K}(U) \quad \hbox{for all} \;\; \l\in[a,b],
$$
where $\mathcal{K}(U)$ denotes the ideal of compact linear operators of $U$. By Theorem 2.1 of Fitzpatrick and Pejsachowicz \cite{FP2}, every Fredholm path admits a parametrix, $\mathfrak{P}$. Thus, the parity of the path $\mathfrak{L}\in\mathscr{C}([a,b],\Phi_{0}(U))$ can be defined as
\begin{equation*}
  \sigma(\mathfrak{L},[a,b])=\deg(\mathfrak{P}(a)\circ\mathfrak{L}(a))
   \deg(\mathfrak{P}(b)\circ\mathfrak{L}(b))
\end{equation*}
where $\mathfrak{P}\in\mathcal{C}([a,b],GL(U))$ is any parametrix of $\mathfrak{L}$ and $\deg\equiv \deg_{LS}$ denotes the Leray--Schauder degree. These geometrical and analytical definitions of the parity coincide. However, the geometrical one shows that, actually, the parity detects any change of orientability since each transversal crossing with $\mathcal{S}(U)$ entails a \textit{change of side} with respect to $\mathcal{S}(U)$ when $GL(U)$ consists of two path-connected components. So, the curve crosses $\mc{S}(U)$ from one of these two components to the other.
\par
Throughout the rest of this  paper, an homotopy $H\in\mathcal{C}([0,1]\times[a,b],\Phi_{0}(U))$  is said to be  \emph{admissible} if $H([0,1]\times\{a,b\})\subset GL(U)$, and two given paths, $\mathfrak{L}_{1}$ and $\mathfrak{L}_{2}$, are said to be $\mathcal{A}$-\emph{homotopic} if they are homotopic through some  admissible homotopy. A fundamental property of the parity, already established by Fitzpatrick and Pejsachowiz \cite{FP2}, is  its invariance under admissible homotopies.
\par
The next result, which is Theorem 4.5 of \cite{JJ},
shows how the parity of any admissible Fredholm path
$\mathfrak{L}\in\mathscr{C}([a,b], \Phi_{0}(U))$ can be computed through the generalized algebraic multiplicity $\chi$.

\begin{theorem}
\label{th4.1}
Any continuous admissible path $\mathfrak{L}\in\mathscr{C}([a,b],\Phi_{0}(U))$ is $\mathcal{A}$-homotopic to an admissible analytic Fredholm curve $\mathfrak{L}_{\omega}\in\mathscr{H}([a,b],\Phi_{0}(U))$. Moreover,
\begin{equation*}
	\sigma(\mathfrak{L},[a,b])=(-1)^{\sum_{i=1}^{n}\chi[\mathfrak{L}_{\omega},\lambda_{i}]},
\end{equation*}
	where
\begin{equation*}
	\Sigma(\mathfrak{L}_{\omega})=\{\lambda_{1},\lambda_{2},...,\lambda_{n}\}.
\end{equation*}
\end{theorem}

Subsequently, for every $\mathfrak{L}\in\mathcal{C}([a,b],\Phi_{0}(U))$ and any isolated eigenvalue $\lambda_{0}\in\Sigma(\mathfrak{L})$, we define the \textit{localized parity} of $\mathfrak{L}$ at $\lambda_{0}$ by
\begin{equation*}
\sigma(\mathfrak{L},\lambda_{0}):=\lim_{\eta\downarrow 0}\sigma(\mathfrak{L},[\lambda_{0}-\eta,\lambda_{0}+\eta]).
\end{equation*}
As a consequence of Theorem \ref{th4.1}, the next result holds (see \cite[Cor. 4.6]{JJ}).

\begin{corollary}
	\label{cr4.2}
Assume  $\mathfrak{L}\in\mathcal{A}_{\lambda_{0}}([a,b],\Phi_{0}(U))$, i.e.,
$\mathfrak{L}\in\mathcal{C}^{r}([a,b],\Phi_{0}(U))$ with  $\lambda_{0}\in\Alg_{\kappa}(\mathfrak{L})$ for some integer $r\geq 1$ and $1\leq \kappa \leq r$. Then,
\begin{equation}
\label{iv.1}
	\sigma(\mathfrak{L},\lambda_{0})=(-1)^{\chi[\mathfrak{L},\lambda_{0}]}.
\end{equation}
\end{corollary}

The  identity \eqref{iv.1} establishes a sharp connection between the topological notion of parity and the algebraic concept of multiplicity. Since the localized parity detects any change of orientation, \eqref{iv.1} does make intrinsic to the concept of algebraic multiplicity any change of orientation. This is rather coherent with the fact that $\chi$ measures how the curve $\mathfrak{L}$ intersects $\mathcal{S}(U)$ through the concept of local intersection index. As the changes of orientation are materialized by any transversal crossing with $\mc{S}(U)$, it is reasonable to think that the
intersection cardinal is odd in all these cases. Here relies precisely the relevance of Corollary \ref{cr4.2}. As a rather direct consequence of \eqref{iv.1}, it follows from Theorem \ref{th32} that the parity also is closely related to the local intersection index through the identity
\begin{equation}\label{C}
\sigma(\mathfrak{L},\l_{0})= (-1)^{i(\mathcal{D}_{M}^{-1}(0),\lambda I_{M} - \mathscr{L}_{\l_{0}}(\mathfrak{L}^{\omega}_{\mathbb{C}});\lambda_{0} I_{M}-\mathscr{L}_{\l_{0}}(\mathfrak{L}^{\omega}_{\mathbb{C}}))},
\end{equation}
where, given  any analytic curve, $\mathfrak{L}^{\omega}$,  $\mathcal{A}$-homotopic to $\mathfrak{L}$,
$\mathscr{L}_{\l_{0}}(\mathfrak{L}^{\omega}_{\mathbb{C}})\in\mathcal{L}(\mathbb{C}^{M})$ stands for the Schur reduction of the complexification $\mathfrak{L}^{\omega}_{\mathbb{C}}$ of $\mathfrak{L}^{\omega}$ at $\l_{0}$.
\par
In \cite{JJ2}, the authors axiomatized the parity within the vain
 of Theorem \ref{th24} for $\chi$. Precisely, the following result was established, where,
 for a given interval $I\subset\mathbb{R}$, $\mathcal{H}_{\l_{0}}(I,\Phi_{0}(U))$
stands for the space of analytic curves $\mathfrak{L}:I\to\Phi_{0}(U)$ such that $\mathfrak{L}(\l)\in GL(U)$ for each $\l\neq\l_{0}$.

\begin{theorem}
	\label{th4.3}
	For every $\varepsilon>0$ and $\lambda_{0}\in\mathbb{R}$, there exists a unique $\mathbb{Z}_{2}$-valued map
$$
	\sigma(\cdot,\lambda_{0}):\;\; \mathcal{H}_{\lambda_0}\equiv  \mathcal{H}_{\lambda_0}((\lambda_{0}-\varepsilon,\lambda_{0}+\varepsilon),\Phi_{0}(U))
	\longrightarrow\mathbb{Z}_{2}
$$
	such that:
	\begin{enumerate}
		\item[{\rm (N)}]  \textbf{Normalization:}  $\sigma(\mathfrak{L},\lambda_{0})=1$ if $\mathfrak{L}(\lambda_{0})\in GL(U)$, and there exists a rank one projection $P_{0}\in\mathcal{L}(U)$ such that
\begin{equation*}
		\sigma\left((\lambda-\lambda_{0})P_{0}+I_{U}-P_{0},\l_0\right)=-1.
\end{equation*}

		\item[{\rm (P)}]  \textbf{Product Formula:} For every  $\mathfrak{L}, \mathfrak{M}\in\mathcal{H}_{\lambda_{0}}$,
		\begin{equation*} \sigma(\mathfrak{L}\circ\mathfrak{M},\lambda_{0})=\sigma(\mathfrak{L},\lambda_{0})
		\cdot\sigma(\mathfrak{M},\lambda_{0}).
		\end{equation*}
	\end{enumerate}
	Moreover, for every $\mathfrak{L}\in\mathcal{H}_{\lambda_{0}}$, the parity map is given by
	\begin{equation*}
	\sigma(\mathfrak{L},\lambda_{0})=(-1)^{\chi[\mathfrak{L},\lambda_{0}]}.
	\end{equation*}
\end{theorem}

Next we will  discuss and sharpen the concept of parity to closed curves,  $\sigma(\cdot,\mathbb{S}^{1})$, going back to Fitzpatrick and Pejsachowiz in \cite{FP3}. Precisely, we consider the circle $\mathbb{S}^{1}$ as obtained from $[a,b]$ by identifying $a$ and $b$, and, for every  $\mathfrak{L}\in\mathcal{C}(\mathbb{S}^{1},\Phi_{0}(U))$, we define $\sigma(\mathfrak{L},\mathbb{S}^{1})$ as
\begin{equation*}
\sigma(\mathfrak{L},\mathbb{S}^{1}):=\deg\left(\mathfrak{P}(a)\circ \mathfrak{P}(b)^{-1}\right),
\end{equation*}
where $\mathfrak{P}:[a,b]\to \Phi_{0}(U)$ is any parametrix of $\mathfrak{L}$ and $\deg\equiv \deg_{LS}$ stands for the Leray--Schauder degree. If $\mathfrak{L}\in\mathscr{C}([a,b],\Phi_{0}(U))$ is closed, i.e., $\mathfrak{L}(a)=\mathfrak{L}(b)$, then these two notions of parity coincide. Indeed, given an admissible closed path $\mathfrak{L}\in\mathscr{C}([a,b],\Phi_{0}(U))$ and a parametrix of it, $\mathfrak{P}:[a,b]\to GL(U)$, thanks to the properties of the Leray--Schauder degree and setting  $\mathfrak{L}(a)=\mathfrak{L}(b)\equiv T\in GL(U)$, we find that
\begin{align*}
 \s(\mathfrak{L},[a,b]) & =\deg(\mathfrak{P}(a)\circ \mathfrak{L}(a))\deg(\mathfrak{P}(b)\circ \mathfrak{L}(b))\\
& =\deg(\mathfrak{P}(a)\circ T)\deg(T^{-1}\circ \mathfrak{P}^{-1}(b))\deg(T^{-1}\circ\mathfrak{P}^{-1}(b))\deg(\mathfrak{P}(b)\circ T) \\
& =\deg(\mathfrak{P}(a)\circ \mathfrak{P}^{-1}(b))\deg(T^{-1}\circ T)=\deg(\mathfrak{P}(a)\circ \mathfrak{P}^{-1}(b))=\sigma(\mathfrak{L},\mathbb{S}^{1}),
\end{align*}
as claimed before. As established by Fitzpatrick and Pejsachowiz \cite{FP3}, this new notion
of parity is also homotopy invariant.
\par
When, in addition, $U$ is of Kuiper type, i.e., $GL(U)$ is contractible, then the space $\Phi_{0}(U)$ is path-connected (see, e.g.,  \cite[Pr. 1.3.5]{FP3}). Thus, the Poincar\'{e} group $\pi_{1}(\Phi_{0}(U),T)$ does not depend on the chosen base point $T\in\Phi_{0}(U)$, i.e.,
$$
 \pi_{1}(\Phi_{0}(U),T)\equiv\pi_{1}(\Phi_{0}(U)).
$$
In such case,  one can introduce the map
\begin{equation}
\label{PEUI}
\sigma:\pi_{1}(\Phi_{0}(U))\longrightarrow \mathbb{Z}_{2}, \qquad \sigma([\gamma]):=\sigma(\gamma,\mathbb{S}^{1});
\end{equation}
$\sigma$ is well defined since it is invariant by homotopy and moreover, it defines a group isomorphism.  Thus, $\pi_{1}(\Phi_{0}(U))\simeq \mathbb{Z}_{2}$ if $U$ is of Kuiper type. Hence, the parity map descrives some non trivial features of the topology of $\Phi_{0}(U)$.
\par
We end this section by establishing a link between  the algebraic multiplicity $\chi$ and
the notion of parity for closed curves $\sigma(\cdot,\mathbb{S}^{1})$. By adapting the proof of \cite[Th. 4.5]{JJ}, it is easily seen that, for every $\mathfrak{L}\in\mathcal{C}(\mathbb{S}^{1},\Phi_{0}(U))$, there exists an analytic path  $\mathfrak{L}_{\o}\in\mathcal{H}(\mathbb{S}^{1},\Phi_{0}(U))$, homotopic to $\mf{L}$, with some regular point $\l_{0}\in\mathbb{S}^{1}$, i.e., such that $\mathfrak{L}_{\o}(\l_{0})\in GL(U)$.  Since $\mathfrak{L}_{\o}$ is analytic and has a regular point, by Theorem 4.4.4 of \cite{LG01},
$$
  \Sigma(\mathfrak{L}_{\o})=\{\l_{1},\l_{2},...,\l_{n}\}.
$$
Moreover, by homotopy invariance, $\sigma(\mathfrak{L},\mathbb{S}^{1})=\sigma(\mathfrak{L}_{\o},\mathbb{S}^{1})$.
Thus, reparameterizing the curve $\mathfrak{L}_{\o}$ with an isometry on $\mathbb{S}^{1}$ so that $a\equiv 0$, and denoting, once more, by $\mathfrak{L}_{\o}$ such a re-parametrization,  it follows from Theorem \ref{th4.1} that
\begin{equation*}
\sigma(\mathfrak{L}_{\o},\mathbb{S}^{1})=\sigma(\mathfrak{L}_{\o},[0,1])=(-1)^{\sum_{i=1}^{n}\chi[\mathfrak{L}_{\o},\l_{i}]}.
\end{equation*}
Therefore, the next result holds.

\begin{theorem}
\label{th4.8}
For every $\mathfrak{L}\in \mathcal{C}(\mathbb{S}^{1},\Phi_{0}(U))$,
\begin{equation*}
	\sigma(\mathfrak{L},\mathbb{S}^{1})=(-1)^{\sum_{i=1}^{n}\chi[\mathfrak{L}_{\o},\l_{i}]}
\end{equation*}
where $\mathfrak{L}_{\o}\in\mathscr{H}(\mathbb{S}^{1},\Phi_{0}(U))$ is homotopic to $\mathfrak{L}$ and $\Sigma(\mathfrak{L}_{\o})=\{\l_{1},\l_{2},...,\l_{n}\}$.
\end{theorem}

\section{K-Theory and spectral theory}

\noindent In this section, which can be considered to be the central section of the article, we connect the theory of real vector bundles with nonlinear spectral theory.  More specifically, we will characterize the obstruction detected by the first Stiefel--Whitney class of vector bundles through spectral properties by means of the Atiyah--J\"anich map.
\par
We begin by recalling some basic facts to introduce the Atiyah--J\"anich morphism. Two real vector bundles, $E$ and $F$, over a compact path-connected topological space $X$, are said to be \emph{stably equivalent} if there are  $N$, $M\in\mathbb{N}$ such that
\begin{equation*}
E\oplus \underline{\mathbb{R}}^{N}\simeq F\oplus \underline{\mathbb{R}}^{M},
\end{equation*}
where $\underline{\mathbb{R}}^{i}$ denotes the trivial bundle $X\times\mathbb{R}^{i}$ of rank $i$ over $X$ for each $i\in\{N,M\}$, $\oplus$ stands for the Whitney sum of vector bundles,  and $\simeq$ expresses that both real vector bundles are isomorphic. Naturally, the stable equivalence induces an equivalence relation in the set  of isomorphism classes of real vector bundles over  $X$,  denoted by $\Vect(X)$, whose associated quotient is the reduced Grothendieck group, $\tilde{K}\mathcal{O}(X)$. Given a real Banach space $U$, the device that connects vector bundle theory and Fredholm operators is the so-called Atiyah--J\"{a}nich index map $$\mathfrak{Ind}:[X,\Phi_{0}(U)]\longrightarrow \tilde{K}\mathcal{O}(X)$$ introduced by Atiyah, \cite{At, Ja},  which is a sort of generalization of the classical notion of index of a Fredholm operator. The notation $[X,\Phi_{0}(U)]$ stands for the set of homotopy classes of continuous maps $X\to \Phi_{0}(U)$. This map is a homomorphism of groups and makes exact the sequence
\begin{equation*}
[X,GL(U)]\overset{i_{\ast}}{\longrightarrow} [X,\Phi_{0}(U)]\overset{\mathfrak{Ind}}{\longrightarrow}\tilde{K}\mathcal{O}(X)
\end{equation*}
where $i_{\ast}$ is the canonical inclusion, i.e.,
$[X,GL(U)]=\Ker \mathfrak{Ind}$. Some reasonably self-contained references for these materials are Mukherjee \cite[Ch. 1,2]{Mu}, Cohen \cite{Cohen}, Zaidenberg et al. \cite{ZaKr} and Husemoller \cite{HM}.

As we need this morphism to be an isomorphism, we describe the real Banach spaces $U$ for which this happens.  We say that that a real Banach space $U$ admits a symmetric basis if there exists an unconditional basis $\{u_{n}\}_{n\in\mathbb{N}}$ on $U$ such that the following two conditions are satisfied:
\begin{enumerate}
	\item For any two sequence of scalars $\{\alpha_{n}\}_{n\in\mathbb{N}}$, $\{\beta_{n}\}_{n\in\mathbb{N}}\subset \mathbb{R}$ such that $|\beta_{n}|\leq |\alpha_{n}|$ for each $n\in\mathbb{N}$, if the series $\sum_{n\in\mathbb{N}}\alpha_{n}u_{n}$ converges, then $\sum_{n\in\mathbb{N}} \beta_{n}u_{n}$ also converges and
	\begin{equation*}
	\left\|\sum_{n\in\mathbb{N}}\beta_{n}u_{n}\right\|\leq \left\|\sum_{n\in\mathbb{N}}\alpha_{n}u_{n}\right\|.
	\end{equation*}
	\item If $\{\alpha_{n}\}_{n\in\mathbb{N}}\subset\mathbb{R}$ is a sequence such that $\sum_{n\in\mathbb{N}}\alpha_{n}u_{n}$ converges, then necessarily $\sum_{n\in\mathbb{N}}\alpha_{n}u_{\sigma(n)}$ also converges for any permutation $\sigma:\mathbb{N}\to\mathbb{N}$ and 
	\begin{equation*}
	\left\|	\sum_{n\in\mathbb{N}}\alpha_{n}u_{n}\right\|=\left\|\sum_{n\in\mathbb{N}}\alpha_{n}u_{\sigma(n)}\right\|.
	\end{equation*}
\end{enumerate}
Under this definition, we can state the following result whose proof can be found in Zaidenberg et al. \cite[Theorem 2.3]{ZaKr}.

\begin{theorem}\label{T10.1.111}
	Let $X$ be a compact path-connected topological space and $U$ be a real Banach space of Kuiper type that contains a complemented infinite dimensional subspace admitting a symmetric basis. Then the index map
	\begin{equation*}
	\mathfrak{Ind}:[X,\Phi_{0}(U)]\longrightarrow \tilde{K}\mathcal{O}(X)
	\end{equation*}
	is an isomorphism of groups.
\end{theorem}

In what follows, we will say that a real Banach space $U$ is \textit{admissible} if it is of Kuiper type and it contains a complemented infinite dimensional subspace admitting a symmetric basis. For example, every real infinite dimensional separable Hilbert space is admissible. Along this section, $U$ will always denote an admissible real Banach space.

\subsection{Orientability of Vector Bundles} A real vector bundle $E\to X$ is said to be orientable if  it admits a trivializing atlas
whose transition functions have positive determinant. This property can be judged by means of characteristic classes, or through the determinant line bundle, being both approaches equivalent.
According to, e.g., Husemoller \cite[Th. 12.1]{HM}, the first  Stiefel--Whitney class of $E$, $\omega_1(E)\in H^{1}(X,\mathbb{Z}_{2})$, is zero if and only if the real bundle $E$ is orientable, where $H^{1}(X,\mathbb{Z}_{2})$ is the first cohomology group of $X$ with coefficients in $\mathbb{Z}_{2}$.
It turns out that a vector bundle $E$ is orientable if and only if its associated determinant line bundle, $\det E:=\wedge^{n} E$, is trivial.
\par
Since, for every continuous map $h:X\to\Phi_{0}(U)$, $\mathfrak{Ind}[h]\in\tilde{K}\mathcal{O}(X)$ is a stable equivalence class of vector bundles, we should make precise what
orientability means for these vector bundles. As  the first Stiefel--Whitney class, $\omega_{1}:\Vect(X)\to H^{1}(X,\mathbb{Z}_{2})$, depends only on the stable equivalence class, it induces, in a natural way, a morphism
\begin{equation*}
w_{1}:\tilde{K}\mathcal{O}(X)\longrightarrow H^{1}(X,\mathbb{Z}_{2}).
\end{equation*}
Thus, it is rather natural to agree that the class $E\in\tilde{K}\mathcal{O}(X)$ is orientable if $w_{1}(E)=0$. In particular, for every class $E\in\tilde{K}\mathcal{O}(X)$, the first Stiefel--Whitney class can be seen a morphism $\pi_{1}(X)\to \mathbb{Z}_{2}$. Indeed, as a consequence of the universal coefficient theorem, the next group isomorphism can be easily established
$$
H^{1}(X,\mathbb{Z}_{2})\simeq \Hom(\pi_{1}(X),\mathbb{Z}_{2}),
$$
i.e., one can identify $H^{1}(X,\mathbb{Z}_{2})$ with $\Hom(\pi_{1}(X),\mathbb{Z}_{2})$. Let us describe this isomorphism. Following \cite[Sect. 2]{FP4} and denoting $\pi_{1}(X)\equiv\pi_{1}(X,x_{0})$, each homomorphism $\varphi:\pi_{1}(X)\to\mathbb{Z}_{2}$ sends the commutator of $\pi_{1}(X)$, $[\pi_{1}(X),\pi_{1}(X)]$, to zero $0\in\mathbb{Z}_{2}$. Thus, it induces a homomorphism
$$
\tilde{\varphi}:\frac{\pi_{1}(X)}{[\pi_{1}(X),\pi_{1}(X)]}\simeq H_{1}(X,\mathbb{Z}_{2})\longrightarrow \mathbb{Z}_{2}.
$$
By the universal coefficient theorem, each $\tilde{\varphi}:H_{1}(X,\mathbb{Z}_{2})\to\mathbb{Z}_{2}$ corresponds to a unique cohomology class $w\in H^{1}(X,\mathbb{Z}_{2})$. The inverse isomorphism $\Gamma:H^{1}(X,\mathbb{Z}_{2})\to \Hom(\pi_{1}(X),\mathbb{Z}_{2})$ can be described explicitly as follows. If $w\in H^{1}(X,\mathbb{Z}_{2})$, and $\gamma\in\pi_{1}(X)$ is represented by $g:\mathbb{S}^{1}\to X$, then
\begin{equation}
\label{UCT}
[\Gamma(w)](\gamma)=\langle g^{\ast}(w), [\mathbb{S}^{1}]\rangle_{\mathbb{Z}_{2}},
\end{equation}
where $g^{\ast}:H^{1}(X,\mathbb{Z}_{2})\to H^{1}(\mathbb{S}^{1},\mathbb{Z}_{2})$ is the induced morphism in cohomology, $[\mathbb{S}^{1}]$ is the generator of $H^{1}(\mathbb{S}^{1},\mathbb{Z}_{2})$, and
$$
\langle -,- \rangle_{\mathbb{Z}_{2}}: H_{1}(-,\mathbb{Z}_{2})\times H^{1}(-,\mathbb{Z}_{2})\to \mathbb{Z}_{2}
$$
is the Kronecker (duality) pairing.
\par
For every $E\in\mathfrak{Ind}([X,\Phi_{0}(U)])=\tilde{K}\mc{O}(X)$, one can also describe the orientability of $E$ in terms of its determinant bundle. Indeed, according to  Wang \cite{W}, for every $h\in [X,\Phi_{0}(U)]$, the determinant bundle of $\mathfrak{Ind}[h]$ can be defined as the line bundle
\begin{equation*}
\det \mathfrak{Ind}[h]:= \wedge^{\max} \Ker h \otimes (\wedge^{\max} \text{coKer} \ h)^{\ast}
\end{equation*}
where $\wedge^{\max}$ denotes the wedge product in the corresponding dimension of the vector space where we are defining the operation.
\par
The equivalence of these two (and some other) notions of orientability is far from evident, and it has been one of the central issues in topological degree theory for Fredholm operators. Let us explain briefly the situation.  The key feature supporting the Leray--Schauder degree, \cite{LS}, is the fact that the space of invertible operators $T:U\to U$ that are compact perturbations of the identity map, denoted by $GL_{K}(U)$, consists of two path-connected components. Thanks to this feature, the Leray--Schauder degree can be constructed in a rather similar way as the classical Brouwer degree \cite{Br} for continuous Euclidean maps. Although, at least a priori, the construction of a degree for Fredholm operators of index zero, should take into account the topological structure of the set of invertible operators, $GL(U)$, Kuiper \cite{K} established in 1965 that $GL(H)$ is contractible, and hence path-connected, for every real or complex infinite dimensional separable Hilbert space $H$. As a consequence of this result, the degree for Fredholm operators cannot be a $\mathbb{Z}$-valued degree for \textit{all} Fredholm maps. So, establishing a huge conceptual difference with respect to the Leray--Schauder degree.
\par
The first ideas to overcame this serious technical difficulty go back to Murkherjea \cite{MU} and Elworthy and Tromba \cite{ET}, where a $\mathbb{Z}$-valued degree for Fredholm maps was introduced by restricting the class of admissible maps. After these pioneering works, many other attempts came in that direction, until Fitzpatrick, Pejsachowicz and Rabier \cite{FP2,FPRa,FPRb,PR} constructed  a $\mathbb{Z}$-valued degree for Fredholm maps that seems to give a definitive answer to this problem. The main conceptual contribution of Fitzpatrick, Pejsachowicz and Rabier was to restrict the class of the admissible Fredhom maps to the \textit{orientable} ones. An orientable Fredholm map is a Fredholm map, $f:\O\subset U\to U$, where $\O$ is an open bounded domain of $U$,  for which its differential, $Df:\Omega\to\Phi_{0}(U)$, is an orientable map. For a given path connected topological space $X$, a continuous map $h:X\to\Phi_{0}(U)$ is said to be orientable in the sense of Fitzpatrick, Pejsachowicz and Rabier ($\mathfrak{F}$-orientable for short) if
\begin{equation*}
\sigma(h\circ \gamma,\mathbb{S}^{1})=1, \quad \text{ for all } [\gamma]\in \pi_{1}(X).
\end{equation*}
Roughly spoken, a map $h:X\to\Phi_{0}(U)$ is orientable if the subset $h(X)\cap GL(U)\subset \Phi_{0}(U)$ behaves as if it would consist of two path-connected components, like in the classical setting, so allowing the definition of a topological degree for these maps.
\par
More recently, in 1998, Benevieri and Furi constructed in \cite{BF1,BF2,BF3} another topological degree for oriented Fredholm maps based on another, rather algebraic, concept of orientability for maps $h:X\to\Phi_{0}(U)$, which will be subsequently refereed to as $\mathfrak{B}$-orientability. In 2005, Wang \cite{W} endowed the algebraic concept of orientability of Benevieri and Furi with a geometrical meaning, through the determinant bundle of the image of the index map of Atiyah--J\"{a}nich. Finally, J. Pejsachowicz \cite{JP} proved that all these notions of orientation do actually coincide. Next,  we will state a theorem collecting the main results of \cite{JP} adapted to our notations here. Besides it establishes the equivalence of the $\mathfrak{F}$-orientability and the $\mathfrak{B}$-orientability, it states that each of this notions of orientability for maps $X\to\Phi_{0}(U)$ corresponds, via the index map, to the notions of orientability for classes $E\in\tilde{K}\mathcal{O}(X)$ discussed previously.

\begin{theorem}
	\label{th1.1}
	Let $U$ be an admissible real Banach space, $X$ be a compact path-connected space, and $h:X\to\Phi_{0}(U)$ a continuous map. Then,
	\begin{enumerate}
		\item[{\rm (a)}]  $h$ is $\mathfrak{F}$-orientable if and only if $w_{1}(\mathfrak{Ind}[h])=0$ in $H^{1}(X,\mathbb{Z}_{2})$.
		\item[{\rm (b)}] $h$ is $\mathfrak{B}$-orientable if and only if $\det \mathfrak{Ind}[h]$ is a trivial line bundle.
	\end{enumerate}
	Moreover, $w_{1}(\mathfrak{Ind}[h])=0$ if and
	only if $\det \mathfrak{Ind}[h]$ is a trivial line bundle. Therefore, the next properties are equivalent: \begin{itemize}
		\item $h$ is $\mathfrak{F}$-orientable;
		\item $h$ is $\mathfrak{B}$-orientable;
		\item $\mathfrak{Ind}[h]$ is an orientable vector bundle.
	\end{itemize}
\end{theorem}

Therefore, there is a single notion of orientability for maps $X\to\Phi_{0}(U)$. Namely, a map $h:X\to \Phi_{0}(U)$ is orientable (in any sense, $\mathfrak{F}$ or $\mathfrak{B}$) if and only if $\mathfrak{Ind}[h]$ is orientable. Moreover, by Theorem \ref{th1.1}, the notion of orientability defined through the Stiefel--Whitney class and the determinant bundle do coincide.

\subsection{Intersection Morphism} In this subsection we study the obstruction detected by the first Stiefel--Whitney fundamental class and its relationship with spectral theory and the intersection index. For any given vector bundle $E$ over $X$,
we will show how to study its topological properties by using some techniques from spectral theory and algebraic geometry in the classifying space $\Phi_{0}(U)$. More precisely, given a vector bundle $E\to X$, the preimage of this bundle via the index map $\mathfrak{Ind}^{-1}(E): X\to\Phi_{0}(U)$ is a parametrized family of Fredholm operators of index zero. Thus, one should be able to reformulate the topological invariants of $E$ in terms of algebraic data using techniques of nonlinear spectral theory (as, e.g., the algebraic multiplicity,  the intersection index, or the parity). Adopting this methodology, we will relate the first Stiefel--Whitney class, $\omega_{1}$, with the concept of algebraic multiplicity and the underlying notion of intersection index. Subsequently, for notational simplicity, for every $\mathfrak{L}\in\mathcal{C}(\mathbb{S}^{1},\Phi_{0}(U))$, we will denote
\begin{equation*}
\chi_{2}[\mathfrak{L},\mathbb{S}^{1}] :=(-1)^{\chi}, \quad i_{2}(\mathfrak{L},\mathbb{S}^{1}):=(-1)^{\mathfrak{E}},
\end{equation*}
where
\begin{equation*}
\chi:=\sum_{\l_{0}\in\Sigma(\mathfrak{L}^{\omega})}\chi[\mathfrak{L}^{\omega},\l_{0}], \quad \mathfrak{E}:=\sum_{\l_{0}\in\Sigma(\mathfrak{L}^{\omega})}i(\mathcal{D}^{-1}(0),\l  I-\mathscr{L}_{\lambda_{0}},\l_0 I-\mathscr{L}_{\l_{0}}),
\end{equation*}
$\mathfrak{L}^{\o}\in\mathscr{H}(\mathbb{S}^{1},\Phi_{0}(U,))$ is any admissible analytic curve homotopic to $\mathfrak{L}$, and $\mathscr{L}_{\l_{0}}\equiv \mathscr{L}_{\l_{0}}[\mathfrak{L}^{\o}_{\mathbb{C}}]$ is the Schur reduction at $\l_0$ of the complexification, $\mathfrak{L}^{\o}_{\mathbb{C}}$, of $\mathfrak{L}^{\o}$. By \eqref{C} and Theorem \ref{th4.8}, it becomes apparent that
$$
  \chi_{2}[\mathfrak{L},\mathbb{S}^{1}]=i_{2}(\mathfrak{L},\mathbb{S}^{1})
$$
for every continuous path $\mathfrak{L}:\mathbb{S}^{1}\to\Phi_{0}(U)$.
\par
To illustrate how the topological properties are related to the spectral ones, we introduce the intersection morphism
\begin{equation*}
\mathfrak{I}:\tilde{K}\mathcal{O}(X)\longrightarrow \Hom(\pi_{1}(X),\mathbb{Z}_{2}),
\end{equation*}
where, for every $E \in \tilde{K}\mathcal{O}(X)$, the map $\mathfrak{I}[E]:\pi_{1}(X)\to \mathbb{Z}_{2}$ is given by
\begin{align*}
\mathfrak{I}[E]([\gamma]):=\chi_{2}[\mathfrak{Ind}^{-1}[E]\circ\gamma,\mathbb{S}^{1}]=i_{2}(\mathfrak{Ind}^{-1}[E]\circ \gamma,\mathbb{S}^{1}).
\end{align*}
The next result establishes that the morphism $\mathfrak{I}$ equals the first Stiefel--Whitney class morphism $w_{1}:\tilde{K}\mathcal{O}(X)\to H^{1}(X,\mathbb{Z}_{2})$ modulo the isomorphism $\Gamma:H^{1}(X,\mathbb{Z}_{2})\to \Hom(\pi_{1}(X),\mathbb{Z}_{2})$ defined in \eqref{UCT}. Since it fully describes the obstruction of the first Stiefel--Whitney class by means of the algebraic multiplicity and the intersection index, it can be considered as one of the main findings of this article.

\begin{theorem}
\label{th7.2}
Under the previous general assumptions, we have the equality
$$
  \mathfrak{I}= \Gamma\circ w_{1}, \quad \hbox{ as morphisms } \tilde{K}\mathcal{O}(X)\longrightarrow\Hom(\pi_{1}(X),\mathbb{Z}_{2}),
$$
where $\Gamma:H^{1}(X,\mathbb{Z}_{2})\to \Hom(\pi_{1}(X),\mathbb{Z}_{2})$ is the isomorphism defined by   \eqref{UCT}.
\end{theorem}
\begin{proof}
	By \cite[Pr. 2.7]{FP4}, we can factorize the Stiefel--Whitney morphism $w_{1}$ as
	\begin{center}
		\begin{tikzcd}
		\tilde{K}\mathcal{O}(X) \arrow{r}{w_{1}} & H^{1}(X,\mathbb{Z}_{2}) \arrow{d}{\Gamma} \\
		\left[X,\Phi_{0}(U)\right] \arrow{u}{\mathfrak{Ind}} \arrow{r}{\tilde{\sigma}} & \Hom(\pi_{1}(X),\mathbb{Z}_{2})
		\end{tikzcd}
	\end{center}
where the morphism $\tilde\sigma$ is defined by
\begin{equation*}
\tilde\sigma: [X,\Phi_{0}(U)]\longrightarrow \Hom(\pi_{1}(X),\mathbb{Z}_{2}), \quad \tilde{\sigma}[h]([\gamma]):=\sigma(h\circ \gamma, \mathbb{S}^{1}),\quad [\gamma]\in\pi_{1}(X).
\end{equation*}
To show the commutativity of the diagram, it suffices to prove that $$\Gamma(w_{1}(\mathfrak{Ind}[h]))=\tilde{\sigma}[h]$$ regarded as homomorphisms $\pi_{1}(X)\to\mathbb{Z}_{2}$.  Since they are $\mathbb{Z}_{2}$-valued group homomorphisms, that equality follows from the equality of their corresponding kernels.  As
$$
w_{1}:\tilde{K}\mathcal{O}(-)\longrightarrow H^{1}(-,\mathbb{Z}_{2}), \quad \mathfrak{Ind}:[-,\Phi_{0}(U)]\longrightarrow \tilde{K}\mathcal{O}(-),
$$
are natural transformations in the category $\mathfrak{Top}$ of topological spaces and continuous maps, it follows that for each continuous $g:\mathbb{S}^{1}\to X$,
\begin{equation}
\label{8.4.7}
g^{\ast}(w_{1}(\mathfrak{Ind}[h]))=w_{1}(g!(\mathfrak{Ind}[h]))=w_{1}(\mathfrak{Ind}[h\circ g]),
\end{equation}
where $g^{\ast}:H^{1}(X,\mathbb{Z}_{2})\to H^{1}(\mathbb{S}^{1},\mathbb{Z}_{2})$ and $g!:\tilde{K}\mathcal{O}(X)\to\tilde{K}\mathcal{O}(\mathbb{S}^{1})$ are the induced morphism by the cohomology functor $H^{1}(-,\mathbb{Z}_{2})$ and the $K$-theory functor $\tilde{K}\mathcal{O}(-)$, respectively. Pick a loop $[\gamma]\in\Ker[\Gamma(w_{1}(\mathfrak{Ind}[h]))]\subset\pi_{1}(X)$. Then if $g:\mathbb{S}^{1}\to X$ is a representation of the loop $[\gamma]$, by \eqref{UCT} and \eqref{8.4.7}, we can deduce that
\begin{equation*}
0=[\Gamma(w_{1}(\mathfrak{Ind}[h]))]([\gamma])=\langle g^{\ast}(w_{1}(\mathfrak{Ind}[h])),[\mathbb{S}^{1}]\rangle_{\mathbb{Z}_{2}}=\langle w_{1}(\mathfrak{Ind}[h\circ g]),[\mathbb{S}^{1}]\rangle_{\mathbb{Z}_{2}}.
\end{equation*}
As
$$
H_{1}(\mathbb{S}^{1},\mathbb{Z}_{2})\simeq \mathbb{Z}_{2}\simeq H^{1}(\mathbb{S}^{1},\mathbb{Z}_{2}),
$$
we find that $\langle w_{1}(\mathfrak{Ind}[h\circ g]),[\mathbb{S}^{1}]\rangle_{\mathbb{Z}_{2}}=0$ in $\mathbb{Z}_{2}$ if and only if $w_{1}(\mathfrak{Ind}[h\circ g])=0$ in $H^{1}(\mathbb{S}^{1},\mathbb{Z}_{2})$.  On the other hand, it is well known that
\begin{equation*}
\tilde{K}\mathcal{O}(\mathbb{S}^{1})=\{[T\mathbb{S}^{1}],[\mathcal{M}]\}\simeq \mathbb{Z}_{2},
\end{equation*}
where $T\mathbb{S}^{1}$ is the tangent bundle of $\mathbb{S}^{1}$ and $\mathcal{M}$ is the M\"{o}bius bundle. Since $\mathcal{M}$ is not orientable, a bundle over $\mathbb{S}^{1}$ is orientable if and only if it is trivial.  This implies that $w_{1}:\tilde{K}\mathcal{O}(\mathbb{S}^{1})\to H^{1}(\mathbb{S}^{1},\mathbb{Z}_{2})$ is an isomorphism. Consequently, $w_{1}(\mathfrak{Ind}[h\circ g])=0$ in $H^{1}(\mathbb{S}^{1},\mathbb{Z}_{2})$ if and only if $\mathfrak{Ind}[h\circ g]$ is the identity on $\tilde{K}\mathcal{O}(\mathbb{S}^{1})$. Since $[\mathbb{S}^{1},GL(U)]=\Ker \mathfrak{Ind}$, necessarily $h\circ g\in[\mathbb{S}^{1},GL(U)]$. Hence, it becomes apparent that $\sigma(h\circ g,\mathbb{S}^{1})=1$, i.e.,
\begin{equation*}
\tilde{\sigma}[h]([\gamma])=\sigma(h\circ g,\mathbb{S}^{1})=1.
\end{equation*}
Therefore, $[\gamma]\in\Ker[\tilde{\sigma}[h]]$. The converse is identical. Hence, we establish the identity
$$
\Ker[\Gamma(w_{1}(\mathfrak{Ind}[h]))]=\Ker[\tilde{\sigma}[h]], \quad h\in[X,\Phi_{0}(U)].
$$
Thus, $\Gamma(w_{1}(\mathfrak{Ind}[h]))=\tilde{\sigma}[h]$ for every $h\in[X,\Phi_{0}(U)]$ and the commutativity of the diagram holds.
Thus,  $\Gamma\circ w_{1}=\tilde\sigma\circ\mathfrak{Ind}^{-1}$. Observe that $\tilde\sigma\circ\mathfrak{Ind}^{-1}=\mathfrak{I}$ since for every  $E\in\tilde{K}\mathcal{O}(X)$ and $[\gamma]\in\pi_{1}(X)$, it follows from Theorem \ref{th4.8} that
\begin{equation*}
\mathfrak{I}[E]([\gamma])=\chi_{2}[\mathfrak{Ind}^{-1}[E]\circ\gamma,\mathbb{S}^{1}]=\sigma(\mathfrak{Ind}^{-1}[E]\circ \gamma,\mathbb{S}^{1})=\tilde{\sigma}[\mathfrak{Ind}^{-1}[E]]([\gamma]).
\end{equation*}
Therefore, $\mathfrak{I}=\Gamma\circ w_{1}$ as claimed. This ends the proof.
\end{proof}
As a consequence of the proof of Theorem \ref{th7.2} we obtain that the intersection morphism $\mathfrak{I}:\tilde{K}\mathcal{O}(X)\to \Hom(\pi_{1}(X),\mathbb{Z}_{2})$ factorizes the diagram
	\begin{center}
	\begin{tikzcd}
	\tilde{K}\mathcal{O}(X) \arrow{r}{w_{1}} \arrow{rd}{\mathfrak{I}} & H^{1}(X,\mathbb{Z}_{2}) \arrow{d}{\Gamma} \\
	\left[X,\Phi_{0}(U)\right] \arrow{u}{\mathfrak{Ind}} \arrow{r}{\tilde{\sigma}} & \Hom(\pi_{1}(X),\mathbb{Z}_{2})
	\end{tikzcd}
\end{center}
As the morphism $\mathfrak{I}$ describes the class $w_{1}$ in terms of the algebraic multiplicity and the intersection index of algebraic varieties, it establishes a hidden connection between the topological information of the vector bundle $E$ and the spectral properties of the underlying Fredholm paths.
\par
Moreover, since $\mathfrak{I}=\Gamma\circ w_{1}$, $\mathfrak{I}$ defines a topological invariant of stable equivalence classes of vector bundles. In other words, $E\neq F$ in $\tilde{K}\mathcal{O}(X)$ for any pair of vector bundles $E, F$ over $X$ with $\mathfrak{I}[E]\neq \mathfrak{I}[F]$. So, the next result holds.

\begin{theorem}
\label{th7.3}
	The intersection morphism $\mathfrak{I}:\tilde{K}\mathcal{O}(X)\rightarrow \Hom(\pi_{1}(X),\mathbb{Z}_{2})$ is a topological invariant of stable equivalence classes of real vector bundles with base $X$.
\end{theorem}

Naturally, thanks to Theorem \ref{th7.2}, the orientability of a given vector bundle
$E\to X$ can be characterized in terms of the intersection index. Precisely, the next result holds.

\begin{corollary}
	\label{C7.1.3}
	Let $E\to X$ be a real vector bundle over $X$. Then,  $E$ is orientable if and only if,
	for every closed curve $[\gamma]\in\pi_{1}(X)$,
	\begin{equation*}
	\chi_{2}[\mathfrak{Ind}^{-1}[E]\circ\gamma,\mathbb{S}^{1}]=
	i_{2}(\mathfrak{Ind}^{-1}[E]\circ\gamma,\mathbb{S}^{1})=1,
	\end{equation*}
	i.e., $E$ is orientable if and only if $\mathfrak{I}[E]\equiv 1$, where $1$ is the identity on $\Hom(\pi_{1}(X),\mathbb{Z}_{2})$.
\end{corollary}

\begin{proof}
	The vector bundle $E\to X$ is orientable if and only $w_{1}(E)=0$ in $H^{1}(X,\mathbb{Z}_{2})$. Thus, since $\mathfrak{I}=\Gamma\circ w_{1}$ and $\Gamma$ is an isomorphism, $E$ is orientable if and only if $\mathfrak{I}[E]\equiv 1$.
\end{proof}

The main interest of the previous results relies on the fact that, being $w_{1}[E]$ difficult to compute in general, the intersection morphism can be easily computed in many particular examples as it will become apparent in the next sections. Another important consequence of establishing these connections  is the fact that the real line bundles can be completely described through their spectral properties. Precisely, the next result holds. In the sequel we denote by $\Vect_{1}(X)$ the set of isomorphism classes of line bundles over $X$. In particular $\Vect_{1}(X)$ is a group with the tensor product $\otimes$ of line bundles.
\begin{theorem}
	\label{th5.5}
	The restricted intersection morphism
	\begin{equation*}
	\mathfrak{I}:\Vect_{1}(X)\longrightarrow \Hom(\pi_{1}(X),\mathbb{Z}_{2})
	\end{equation*}
	is an isomorphism. In other words, two line bundles $L, L'\in \Vect_{1}(X)$ are isomorphic if and only if
	\begin{equation*}
	\chi_{2}[\mathfrak{Ind}^{-1}[L]\circ\gamma,\mathbb{S}^{1}]=\chi_{2}[\mathfrak{Ind}^{-1}[L']\circ\gamma,\mathbb{S}^{1}]
	\end{equation*}
	for every $[\gamma]\in\pi_{1}(X)$.
\end{theorem}
\begin{proof}
	Let us start by proving that the restricted morphism $\mathfrak{I}$ is well defined. We need to prove that the isomorphism classes of the line bundles coincide with the stable equivalence ones. If $L, L'\in \Vect_{1}(X)$ are isomorphic, they are clearly stably isomorphic. Let us suppose that $L, L'\in \Vect_{1}(X)$ are stably isomorphic. Then, by the properties of the Stiefel--Whitney class, $\omega_{1}(L)=\omega_{1}(L')$. Since 	$\omega_{1}:\Vect_{1}(X)\to H^{1}(X,\mathbb{Z}_{2})$
	is an isomorphism, necessarily $\omega_{1}(L)=\omega_{1}(L')$ implies that the line bundles $L, L'$ are isomorphic. This proves our claim. On the other hand, we have that
	\begin{equation*}
	\mathfrak{I}=\Gamma\circ \omega_{1}, \quad \text{ as maps } \  \Vect_{1}(X)\longrightarrow \Hom(\pi_{1}(X),\mathbb{Z}_{2}).
	\end{equation*}
	Since both, $\Gamma$ and $\omega_{1}:\Vect_{1}(X)\to H^{1}(X,\mathbb{Z}_{2})$, are isomorphisms, it follows that $\mathfrak{I}:\Vect_{1}(X)\to \Hom(\pi_{1}(X),\mathbb{Z}_{2})$ is an isomorphism. This ends the proof.
\end{proof}

As shown by the following result, in certain cases, the intersection morphism can classify also every stable equivalence class of real vector bundles.

\begin{theorem}
	If all orientable vector bundles are stably trivial or, equivalently, if all orientable maps $X\to\Phi_{0}(U)$ are homotopic, then the intersection morphism
	\begin{equation*}
	\mathfrak{I}:\tilde{K}\mathcal{O}(X)\longrightarrow \Hom(\pi_{1}(X),\mathbb{Z}_{2})
	\end{equation*}
	is an isomorphism. In other words, the intersection morphism $\mathfrak{I}$ classifies all stable equivalence classes of vector bundles, i.e., $E, F\in \Vect(X)$ are stably equivalent if and only if
	\begin{equation*}	
	\chi_{2}[\mathfrak{Ind}^{-1}[E]\circ\gamma,\mathbb{S}^{1}]=
	\chi_{2}[\mathfrak{Ind}^{-1}[F]\circ\gamma,\mathbb{S}^{1}]\quad \hbox{for all} \;\; [\gamma]\in\pi_{1}(X).
	\end{equation*}
\end{theorem}
\begin{proof}
	Since all orientable vector bundles are stably trivial, it follows that $\Ker[w_{1}]$ is the identity on $\tilde{K}\mathcal{O}(X)$. Hence $w_{1}:\tilde{K}\mathcal{O}(X)\to H^{1}(X,\mathbb{Z}_{2})$ is injective. It is also surjective since $\omega_{1}:\Vect_{1}(X)\to H^{1}(X,\mathbb{Z}_{2})$ is an isomorphism. Hence $w_{1}$ is an isomorphism. The proof concludes noticing that $\mathfrak{I}=\Gamma\circ w_{1}$.
\end{proof}
\par
Roughly spoken, the Stiefel--Whitney class (or, equivalently, the intersection morphism) $\mathfrak{I}[E]\equiv(\Gamma\circ w_{1})[E]:\pi_{1}(X)\to\mathbb{Z}_{2}$ measures how the vector bundle $E\to X$ twists along a given loop of the base space $[\gamma]\in\pi_{1}(X)$. In the following section, we will introduce a new topological invariant of vector bundles that will encode all the values of $\mathfrak{I}[E]:\pi_{1}(X)\to\mathbb{Z}_{2}$ in a sort of generalized analogue of the arithmetical mean giving rise to a global measure of the torsion of $E$. This invariant is far more comfortable to work with than with $w_{1}$, a it provides us with a real number instead of a cohomology class.

\section{The global torsion invariant}

This section introduces a new topological invariant of stable equivalence classes of real vector bundles that encodes the information given by the first Stiefel--Whitney class, $w_1$. Besides characterizing the orientability of a vector bundle, $w_{1}$ also gives some useful information on non orientable bundles which can be invoked to classify them. This information is actually encoded in the values of the map $\mathfrak{I}[E]:\pi_{1}(X)\to \mathbb{Z}_{2}$. The basic idea consists in summing
up the values of this map.
\par
For any given closed smooth manifold $M$, there exists a Riemannian metric $g$ defined on $M$  for which $(M,g)$ becomes a Riemannian manifold. In the sequel, we fix this metric $g$ and a base-point $\mathbf{x}\in M$. Then, we define the global torsion invariant $\L:\tilde{K}\mathcal{O}(M)\longrightarrow [-1,1]$ by
\begin{equation*}
\L(E):= \int_{\mathcal{L}_{\mathbf{x}}(M)}\mathfrak{I}[E]([\gamma])\ d\mu_{\mathbf{x}}(\gamma),
\end{equation*}
where $\mathcal{L}_{\mathbf{x}}(M)$ stands for the loop space of $M$, i.e., the space of continuous loops $\gamma:\mathbb{S}^{1}\to M$ with base-point $\mathbf{x}$, $\gamma(0)=\mathbf{x}$, and $\mu_{\mathbf{x}}:\mathcal{B}_{\mathbf{x}}\to [0,+\infty]$ is the normalised Wiener measure on $\mathcal{L}_{\mathbf{x}}(M)$; $\mathcal{B}_{\mathbf{x}}$ denotes the Borel $\sigma$-algebra of $\mathcal{L}_{\mathbf{x}}(M)$ under the topology of the uniform convergence. See Appendix A for the definition of the measure $\mu_{\mathbf{x}}$.
\par
The aim of $\L$ is packaging the information provided by the class $w_{1}$,  in a robust and compact way, summing up over all its possible values. Thanks to Theorem \ref{th7.2}, $\L$
is easily computable in a number of cases. The next result will be very useful for these computations.

\begin{theorem}
	\label{T7.2.1}
	For every $E\in\tilde{K}\mathcal{O}(M)$, $\mathfrak{I}[E]([\cdot])\in L^{1}(\mathcal{L}_{\mathbf{x}}(M),\mu_{\mathbf{x}})$ and
	\begin{equation}
	\label{7.7.2}
	\int_{\mathcal{L}_{\mathbf{x}}(M)}\mathfrak{I}[E]([\gamma]) \ d\mu_{\mathbf{x}}(\gamma)=\sum_{[\eta]\in\pi_{1}(M)}\chi_{2}[\mathfrak{Ind}^{-1}[E]\circ \eta,\mathbb{S}^{1}] \cdot \mu_{\mathbf{x}}([\eta]).
	\end{equation}
\end{theorem}

\begin{proof}
	Since $M$ is a real topological manifold, by Lee \cite[Th. 1.16]{Lee2}, $\pi_{1}(M)$ is countable.
	Thus, there exists a sequence of loops, $\eta_{n}:\mathbb{S}^{1}\to M$, $n\in\Z$, possibly finite, such that
	$$
	\pi_{1}(M)=\biguplus_{n\in\mathbb{Z}}[\eta_{n}],
	$$
	where  $\uplus$ denotes the disjoint union. Since $\mathfrak{I}[E]$ is a map  $\pi_{1}(M)\to \mathbb{Z}_{2}$, it is constant on each homotopy class. Thus, we can rewrite
	\begin{equation*}
	\mathfrak{I}[E]([\gamma])=\sum_{n\in\mathbb{Z}}\mathfrak{I}[E]([\eta_{n}]) \cdot \mathbf{1}_{[\eta_{n}]}(\gamma), \quad \gamma\in\mathcal{L}_{\mathbf{x}}(M),
	\end{equation*}
	where $ \mathbf{1}_{[\eta_{n}]}(\gamma)=1$ if $\g\in [\eta_n]$, and
	$ \mathbf{1}_{[\eta_{n}]}(\gamma)=0$ if not. Hence, it becomes apparent that
	\begin{equation*}
	\mathfrak{I}[E]([\gamma])=\sum_{n\in\mathbb{Z}}\chi_{2}[\mathfrak{Ind}^{-1}[E]\circ\eta_{n},\mathbb{S}^{1}] \cdot \mathbf{1}_{[\eta_{n}]}(\gamma).
	\end{equation*}
	As $\mathfrak{I}[E]([\cdot])$ is the pointwise limit of the simple functions defined by
	\begin{equation*}
	f_{m}(\gamma):=\sum_{n=-m}^{m}\chi_{2}[\mathfrak{Ind}^{-1}[E]\circ\eta_{n},\mathbb{S}^{1}] \cdot \mathbf{1}_{[\eta_{n}]}(\gamma)
	\end{equation*}
	and $[\eta_{n}]\in\mathcal{B}_{\mathbf{x}}$, $\mathfrak{I}[E]([\cdot]):\mathcal{L}_{\mathbf{x}}(M)\to\mathbb{R}$ is measurable. Moreover, since $|\mathfrak{I}[E]([\gamma])|=1$ for all $\gamma\in\mathcal{L}_{\mathbf{x}}(M)$, we have that
	\begin{equation*}
	\int_{\mathcal{L}_{\mathbf{x}}(M)}|\mathfrak{I}[E]([\gamma])|\ d\mu_{\mathbf{x}}(\gamma)=1
	\end{equation*}
	and hence $\mathfrak{I}[E]([\cdot])$ is integrable, i.e., $\mathfrak{I}[E]([\cdot])\in L^{1}(\mathcal{L}_{\mathbf{x}}(M),\mu_{\mathbf{x}})$. Finally, by the dominated convergence theorem,
	we find that
	\begin{align*}
	\int_{\mathcal{L}_{\mathbf{x}}(M)} \mathfrak{I}[E]([\gamma]) \ d\mu_{\mathbf{x}}(\gamma) & =\lim_{m\to\infty}\int_{\mathcal{L}_{\mathbf{x}}(M)}f_{m}(\gamma) \ d\mu_{\mathbf{x}}(\gamma) \\
	&=\lim_{m\to\infty}\sum_{n=-m}^{m}\chi_{2}[\mathfrak{Ind}^{-1}[E]\circ\eta_{n},\mathbb{S}^{1}]\cdot \mu_{\mathbf{x}}([\eta_{n}]) \\
	& =\sum_{[\eta]\in\pi_{1}(M)}\chi_{2}[\mathfrak{Ind}^{-1}[E]\circ\eta,\mathbb{S}^{1}] \cdot \mu_{\mathbf{x}}([\eta]).
	\end{align*}
	This shows \eqref{7.7.2} and ends the proof.
\end{proof}

The next result provides us with a useful criterium, in terms of $\L$, to ascertain whether, or not, a vector bundle $E\to M$ is orientable.

\begin{theorem}\label{T7.2.2}
	A vector bundle $E\to M$ is orientable if and only if
	\begin{equation}
	\label{7.2}
	\L(E)\equiv \int_{\mathcal{L}_{\mathbf{x}}(M)}\mathfrak{I}[E]([\gamma]) \ d\mu_{\mathbf{x}}(\gamma)=1.
	\end{equation}
\end{theorem}

\begin{proof}
	By Corollary \ref{C7.1.3}, $\mathfrak{I}[E]\equiv 1$ if $E$ is orientable. Thus, $\mathfrak{I}[E]([\gamma])=1\in\Z_2$ for all $[\gamma]\in\pi_{1}(M)$, which is equivalent to
	\begin{equation*}
	\chi_{2}[\mathfrak{Ind}^{-1}[E]\circ \gamma,\mathbb{S}^{1}]=1\quad \hbox{for all}\;\;
	[\g] \in \pi_{1}(M).
	\end{equation*}
	Therefore, \eqref{7.7.2} implies that
	\begin{equation*}
	\L(E)\equiv \int_{\mathcal{L}_{\mathbf{x}}(M)}\mathfrak{I}[E]([\gamma]) \ d\mu_{\mathbf{x}}(\gamma)=\sum_{[\eta]\in\pi_{1}(M)}\mu_{\mathbf{x}}([\eta])=\mu_{\mathbf{x}}(\mathcal{L}_{\mathbf{x}}(M))=1,
	\end{equation*}
	since $\mu_{\mathbf{x}}$ is a probability measure.
	\par
	Conversely, if \eqref{7.2} holds, then \eqref{7.7.2} implies that
	\begin{equation}
	\label{7.7.3}
	\sum_{[\eta]\in\pi_{1}(M)}\chi_{2}[\mathfrak{Ind}^{-1}[E]\circ \eta,\mathbb{S}^{1}]\cdot \mu_{\mathbf{x}}([\eta])=1.
	\end{equation}
	Subsequently, we consider the following subsets of $\pi_{1}(M)$:
	\begin{align*}
	\mathscr{P} & :=\left\{[\eta]\in\pi_{1}(M): \chi_{2}[\mathfrak{Ind}^{-1}[E]\circ \eta,\mathbb{S}^{1}]=1\right\},\\[1ex]
	\mathscr{N} & :=\left\{[\eta]\in\pi_{1}(M): \chi_{2}[\mathfrak{Ind}^{-1}[E]\circ \eta,\mathbb{S}^{1}]=-1\right\}.
	\end{align*}
	According to \eqref{7.7.3}, we have that
	\begin{equation*}
	1=\sum_{[\eta]\in\pi_{1}(M)}\chi_{2}[\mathfrak{Ind}^{-1}[E]\circ \eta,\mathbb{S}^{1}]\cdot \mu_{\mathbf{x}}([\eta])=\sum_{[\eta]\in\mathscr{P}}\mu_{\mathbf{x}}([\eta])-
	\sum_{[\eta]\in\mathscr{N}}\mu_{\mathbf{x}}([\eta]).
	\end{equation*}
	On the other hand,
	$$
	1= \sum_{[\eta]\in\pi_{1}(M)}\mu_{\mathbf{x}}([\eta])=
	\sum_{[\eta]\in\mathscr{P}}\mu_{\mathbf{x}}([\eta])+
	\sum_{[\eta]\in\mathscr{N}}\mu_{\mathbf{x}}([\eta]).
	$$
	Thus, by subtracting the last two identities, we find that
	\begin{equation*}
	\sum_{[\eta]\in\mathscr{N}}\mu_{\mathbf{x}}([\eta])=0.
	\end{equation*}
	On the other hand, on the  last two lines before the statement of
	Theorem \ref{T7.2.5} in Section 10.2.1 bellow it is established that
	every path-connected component of $\mathcal{L}_{\mathbf{x}}(M)$ has non vanishing Wiener measure.
	Thus, $\mu_{\mathbf{x}}([\eta])>0$ for each $[\eta]\in\pi_{1}(M)$ and hence, $\mathscr{N}=\emptyset$.
	Therefore,
	$$
	\chi_{2}[\mathfrak{Ind}^{-1}[E]\circ \eta,\mathbb{S}^{1}]=1 \quad\hbox{for all}\;\; [\eta]\in\pi_{1}(M).
	$$
	This implies that $\mathfrak{I}[E]\equiv 1$. By Corollary \ref{C7.1.3}, $E$ is orientable.
\end{proof}

The final result of this section reads as follows.

\begin{theorem}
	\label{T7.2.3}
	The map $\L:\tilde{K}\mathcal{O}(M)\to [-1,1]$ is a topological invariant of stable equivalence classes of real vector bundles over $M$, i.e., $\L(E)=\L(F)$ if $E=F$ in $\tilde{K}\mathcal{O}(M)$.
\end{theorem}
\begin{proof}
	Suppose that $E=F$ in $\tilde{K}\mathcal{O}(M)$. Then, $\mathfrak{I}[E]=\mathfrak{I}[F]$, because $\mathfrak{I}$ is a topological invariant of real vector bundles. Thus,
	\begin{equation*}
	\chi_{2}[\mathfrak{Ind}^{-1}[E]\circ \eta,\mathbb{S}^{1}]= \chi_{2}[\mathfrak{Ind}^{-1}[F]\circ \eta,\mathbb{S}^{1}]\quad \hbox{for all}\;\;  [\eta]\in\pi_{1}(M).
	\end{equation*}
	Hence, we obtain that
	\begin{equation*}
	\sum_{[\eta]\in\pi_{1}(M)}\chi_{2}[\mathfrak{Ind}^{-1}[E]\circ \eta,\mathbb{S}^{1}] \cdot \mu_{\mathbf{x}}([\eta])= \sum_{[\eta]\in\pi_{1}(M)}\chi_{2}[\mathfrak{Ind}^{-1}[F]\circ \eta,\mathbb{S}^{1}] \cdot \mu_{\mathbf{x}}([\eta]).
	\end{equation*}
	Therefore, by \eqref{7.7.2}, $\L(E)=\L(F)$. This ends the proof.
\end{proof}
However, since, by Theorem \ref{T7.2.2}, $\L$  is based on $w_{1}$,  it does not allow to compare orientable vector bundles. Instead, $\L$ is useful to compare non orientable bundles, as it
somehow measures their  degrees of non-orientability.
\par
{We end this section by providing with an additive formulae for $\L$.
	
	\begin{proposition}
		\label{P7.2.4}
		For every $E$, $F\in\tilde{K}\mathcal{O}(M)$,
		\begin{equation*}
		\L([E]\oplus [F])=\sum_{[\eta]\in\pi_{1}(M)}\chi_{2}[\mathfrak{Ind}^{-1}[E]\circ \eta,\mathbb{S}^{1}]\cdot \chi_{2}[\mathfrak{Ind}^{-1}[F]\circ\eta,\mathbb{S}^{1}]\cdot \mu_{\mathbf{x}}([\eta]).
		\end{equation*}
	\end{proposition}
	\begin{proof}
		By definition,
		\begin{equation*}
		\L([E]\oplus [F])=\sum_{[\eta]\in\pi_{1}(M)}\chi_{2}[\mathfrak{Ind}^{-1}([E]\oplus[F])\circ \eta,\mathbb{S}^{1}]\cdot \mu_{\mathbf{x}}([\eta]).
		\end{equation*}
		Moreover, setting
		$$
		(\mathfrak{L}\circ\mathfrak{P})(x):=\mathfrak{L}(x)\circ\mathfrak{P}(x)\quad \hbox{for all}\;\; x\in M,
		$$
		it is apparent that
		$$
		\mathfrak{Ind}^{-1}:(\tilde{K}\mathcal{O}(M),\oplus)\to([M,\Phi_{0}(U)],\circ)
		$$
		defines a homomorphism of groups. Thus,
		\begin{equation*}
		\chi_{2}[\mathfrak{Ind}^{-1}([E]\oplus[F])\circ \eta,\mathbb{S}^{1}]=\chi_{2}[\{\mathfrak{Ind}^{-1}[E]\circ\mathfrak{Ind}^{-1}[F]\}\circ \eta,\mathbb{S}^{1}].
		\end{equation*}
		Hence
		\begin{equation*}
		\L([E]\oplus [F])=\sum_{[\eta]\in\pi_{1}(M)}\chi_{2}[(\mathfrak{Ind}^{-1}[E]\circ \eta)\circ(\mathfrak{Ind}^{-1}[F]\circ \eta),\mathbb{S}^{1}]\cdot \mu_{\mathbf{x}}([\eta]).
		\end{equation*}
		Therefore,  by the definition of $\chi_{2}$ and the product formula of the generalized algebraic multiplicity, we find that
		\begin{equation*}
		\chi_{2}[(\mathfrak{Ind}^{-1}[E]\circ \eta)\circ(\mathfrak{Ind}^{-1}[F]\circ \eta),\mathbb{S}^{1}]=\chi_{2}[\mathfrak{Ind}^{-1}[E]\circ\eta,\mathbb{S}^{1}]\cdot\chi_{2}[\mathfrak{Ind}^{-1}[F]\circ\eta,\mathbb{S}^{1}].
		\end{equation*}
		This concludes the proof.
	\end{proof}
	
	\subsection{Decomposition of the Wiener measure}
	
	\noindent In this section, we will reduce the calculation of the global torsion invariant to the determination of the heat kernel of the universal covering of $M$, which is far more easy to compute than the one of $M$. Note that the loop space of $M$, $\mathcal{L}_{\mathbf{x}}(M)$, can be expressed as the union of its path-connected components. In other words,
	\begin{equation}\label{7.7.5}
	\mathcal{L}_{\mathbf{x}}(M)=\biguplus_{\eta\in\pi_{1}(M,\mathbf{x})}[\eta],
	\end{equation}
	where $[\eta]$ denotes the homotopy class of $\eta$. Let $\tilde{M}$ be the universal covering space of $M$ with covering projection $\pi:\tilde{M}\to M$. Consider in $\tilde{M}$ the Riemannian structure given by the pull-back metric $\tilde{g}:=\pi^{\ast}g$. Then, $\pi:(\tilde{M},\tilde{g})\to(M,g)$ is a regular Riemannian covering. According to, e.g.,
	\cite[Cor. 4, Sect. 6, Ch. 2]{SPN},  the fundamental group of $M$ based on $\mathbf{x}$, $\pi_{1}(M,\mathbf{x})$, is isomorphic to the group of deck, or covering, transformations of the covering $\pi:\tilde{M}\to M$, subsequently denoted by $\Aut_{M}\tilde{M}$. Actually, once chosen $\tilde{\mathbf{x}}\in \pi^{-1}(\mathbf{x})$, the isomorphism can be defined through
	\begin{equation*}
	\Phi: \pi_{1}(M,\mathbf{x}) \longrightarrow \Aut_{M}\tilde{M}, \quad [\eta] \mapsto \varphi_{\eta},
	\end{equation*}
	where $\varphi_{\eta}:\tilde{M}\to\tilde{M}$ is the unique covering transformation that sends $\tilde{\mathbf{x}}$ to $\tilde{\eta}(1)$, and $\tilde{\eta}$ is the unique lifting of $\eta$ with $\tilde{\eta}(0)=\tilde{\mathbf{x}}$. In this way, we can rewrite \eqref{7.7.5} in the form
	\begin{equation*}
	\mathcal{L}_{\mathbf{x}}(M)=\biguplus_{\varphi\in\Aut_{M}\tilde{M}}\mathcal{L}^{\varphi}_{\mathbf{x}}(M),
	\end{equation*}
	where $\mathcal{L}^{\varphi}_{\mathbf{x}}(M)$ stands for the path component of $\mathcal{L}_{\mathbf{x}}(M)$ containing the homotopy class $\Phi^{-1}(\varphi)$. According to
	\cite[Th. 4.3]{B3} and \cite{U}, it is easily seen that the map
	\begin{equation*}
	\Theta: \biguplus_{\varphi\in\Aut_{M}\tilde{M}}\mathcal{C}_{\tilde{\mathbf{x}}}^{\varphi(\tilde{\mathbf{x}})}(\tilde{M}) \longrightarrow \mathcal{L}_{\mathbf{x}}(M), \quad
	\tilde{\eta} \mapsto \pi\circ \tilde{\eta},
	\end{equation*}
	is a homeomorphism with the uniform convergence topology, where we use the notation $\tilde{\eta}$  to emphasize  that the curve is defined on $\tilde{M}$, and the spaces $\mathcal{C}_{\tilde{\mathbf{x}}}^{\varphi(\tilde{\mathbf{x}})}(\tilde{M})$ are defined by
	\begin{equation*}
	\mathcal{C}_{\tilde{\mathbf{x}}}^{\varphi(\tilde{\mathbf{x}})}(\tilde{M}):=\{\gamma\in \mathcal{C}([0,1],\tilde{M}): \gamma(0)=\tilde{\mathbf{x}}, \ \gamma(1)=\varphi(\tilde{\mathbf{x}}) \}.
	\end{equation*} 
	Moreover, it preserves the Wiener measure, in the sense that, for every $B\in\mathcal{B}_{\mathbf{x}}$,
	\begin{equation*}
	\lambda_{\mathbf{x}}(B)=\sum_{\varphi\in\Aut_{M}\tilde{M}}
	\lambda_{\tilde{\mathbf{x}}}^{\varphi(\tilde{\mathbf{x}})}\left(\Theta^{-1}(B)\cap \mathcal{C}_{\tilde{\mathbf{x}}}^{\varphi(\tilde{\mathbf{x}})}(\tilde{M})\right),
	\end{equation*}
	where $\lambda_{\tilde{\mathbf{x}}}^{\varphi(\tilde{\mathbf{x}})}$ is the non-normalised Wiener measure on $\mathcal{C}_{\tilde{\mathbf{x}}}^{\varphi(\tilde{\mathbf{x}})}(\tilde{M})$ and $\lambda_{\mathbf{x}}$ is the non-normalised Wiener measure on $\mathcal{L}_{\mathbf{x}}(M)$ (see Appendix A or \cite{B3} for the precise definition). As a direct consequence, setting $B=\mathcal{ L}_{\mathbf{x}}(M)$, the next relationship between the heat kernel of $M$ and the corresponding heat kernel of its universal covering space $\tilde{M}$ holds
	\begin{equation}\label{7.7.6}
	p_{1}(\mathbf{x},\mathbf{x})=\sum_{\varphi\in\Aut_{M}\tilde{M}}\tilde{p}_{1}(\tilde{\mathbf{x}},\varphi(\tilde{\mathbf{x}})),
	\end{equation}
	where $\tilde{p}_{t}(x,y)$ stands for the heat kernel of $\tilde{M}$. In particular, since
	\begin{equation*}
	\Theta(\mathcal{C}_{\tilde{\mathbf{x}}}^{\varphi(\tilde{\mathbf{x}})}(\tilde{M}))
	=\mathcal{L}^{\varphi}_{\mathbf{x}}(M)\quad  \hbox{for all}\;\; \varphi\in\Aut_{M}\tilde{M},
	\end{equation*}
	the restricted map $\Theta_{\varphi}:\mathcal{C}_{\tilde{\mathbf{x}}}^{\varphi(\tilde{\mathbf{x}})}(\tilde{M})\to \mathcal{L}^{\varphi}_{\mathbf{x}}(M)$ is also a homeomorphism. Moreover, for every  $B\in\mathcal{B}_{\mathbf{x}}\cap\mathcal{L}^{\varphi}_{\mathbf{x}}(M)$, we have that
	\begin{equation}\label{7.7.7}
	\lambda_{\mathbf{x}}(B)=\sum_{\phi\in\Aut_{M}\tilde{M}}
	\lambda_{\tilde{\mathbf{x}}}^{\phi(\tilde{\mathbf{x}})}\left(\Theta^{-1}(B)
	\cap\mathcal{C}_{\tilde{\mathbf{x}}}^{\varphi(\tilde{\mathbf{x}})}(\tilde{M})\right)
	=\lambda_{\tilde{\mathbf{x}}}^{\varphi(\tilde{\mathbf{x}})}(\Theta^{-1}(B)).
	\end{equation}
	Thus, $\Theta_{\varphi}$ also preserves the Wiener measure. As according to  \eqref{7.7.7} we have that
	\begin{equation}
	\label{7.7.8}
	\lambda_{\mathbf{x}}(\mathcal{L}^{\varphi}_{\mathbf{x}}(M))=\lambda_{\tilde{\mathbf{x}}}^{\varphi(\tilde{\mathbf{x}})}(\mathcal{C}_{\tilde{\mathbf{x}}}^{\varphi(\tilde{\mathbf{x}})}(\tilde{M}))=\tilde{p}_{1}(\tilde{\mathbf{x}},\varphi(\tilde{\mathbf{x}}))>0,
	\end{equation}
	it is apparent that every path-connected component of $\mathcal{L}_{\mathbf{x}}(M)$ has non vanishing Wiener measure. Furthermore, as a consequence of \eqref{7.7.6} and \eqref{7.7.8}, the following result holds.
	\begin{theorem}
		\label{T7.2.5}
		The normalized Wiener measure of the path components of $\mathcal{L}_{\mathbf{x}}(M)$ is given by
		\begin{equation}
		\label{7.7.9}
		\mu_{\mathbf{x}}(\mathcal{L}_{\mathbf{x}}^{\varphi}(M))=
		\frac{\tilde{p}_{1}(\tilde{\mathbf{x}},\varphi(\tilde{\mathbf{x}}))}
		{\sum_{\phi}\tilde{p}_{1}(\tilde{\mathbf{x}},\phi(\tilde{\mathbf{x}}))}, \quad \varphi\in\Aut_{M} \tilde{M},
		\end{equation}
		where $\tilde{p}_{t}(x,y)$ denotes the heat kernel of $\tilde{M}$, $\tilde{\mathbf{x}}\in \pi^{-1}(\mathbf{x})$ and the sum runs in $\phi\in\Aut_{M}\tilde{M}$.
	\end{theorem}
	
	As a direct consequence of Theorems \ref{7.7.2} and \ref{T7.2.5}, one can determine the global torsion invariant of any given vector bundle $E\to M$ in terms of the generalized algebraic multiplicity $\chi$ and the heat kernel of the universal covering $\tilde{M}$.
	
	\begin{theorem}
		\label{T7.2.6}
		For every $E\in\tilde{K}\mathcal{O}(M)$,
		\begin{equation*}
		\int_{\mathcal{L}_{\mathbf{x}}(M)}\mathfrak{I}[E]([\gamma]) \ d\mu_{\mathbf{x}}(\gamma)=\frac{\sum_{\varphi}\chi_{2}[\varphi]\cdot \tilde{p}_{1}(\tilde{\mathbf{x}},\varphi(\tilde{\mathbf{x}}))}{\sum_{\varphi}\tilde{p}_{1}(\tilde{\mathbf{x}},\varphi(\tilde{\mathbf{x}}))},
		\end{equation*}
		where both sums run in $\varphi\in\Aut_{M}\tilde{M}$, and
		$$
		\chi_{2}[\varphi]\equiv\chi_{2}[\mathfrak{Ind}^{-1}[E]\circ\Phi^{-1}(\varphi),\mathbb{S}^{1}].
		$$
	\end{theorem}
	
	Moreover, since by Theorem \ref{th5.5}, the restricted morphism $\mathfrak{I}: \Vect_{1}(X)\to \Hom(\pi_{1}(X),\mathbb{Z}_{2})$ is an isomorphism, it follows that
	\begin{equation}
		\label{7.7.10}
	\L(\tilde{K}\mathcal{O}(M))=
	\L(\Vect_{1}(X))=\left\{\int_{\mathcal{L}_{\mathbf{x}}(M)} \xi([\gamma]) \ d\mu_{\mathbf{x}}(\gamma): \xi\in\Hom(\pi_{1}(X),\mathbb{Z}_{2}) \right\}.
	\end{equation}
	Consequently, the values of the global torsion invariant are given by
	\begin{equation}
	\label{7.7.11}
	\L(\tilde{K}\mathcal{O}(M))=\left\{\frac{\sum_{\varphi}\zeta(\varphi)\cdot \tilde{p}_{1}(\tilde{\mathbf{x}},\varphi(\tilde{\mathbf{x}}))}{\sum_{\varphi}\tilde{p}_{1}(\tilde{\mathbf{x}},\varphi(\tilde{\mathbf{x}}))} : \zeta\in \Hom(\Aut_{\tilde{M}}M,\mathbb{Z}_{2}) \right\}.
	\end{equation}
	In the next section, we will see that equation \eqref{7.7.11} is useful for ascertaining the values of the global torsion invariant in some practical examples of interest.

	\section{Examples}
	
	\noindent In this section we will compute the global torsion invariant of the circle $\mathbb{S}^{1}$ and the $n$-dimensional torus $\mathbb{T}^{n}$ by using all the machinery developed in the preceding sections.
	
	\subsection{Global torsion invariant of  $\mathbb{S}^{1}$} The aim of this subsection  is computing  $\L:\tilde{K}\mathcal{O}(\mathbb{S}^{1})\to[-1,1]$ where  the circle $\mathbb{S}^{1}$ is regarded as the quotient $\mathbb{R}/2\sqrt{\pi}\mathbb{Z}$; the period factor $2\sqrt{\pi}$ is chosen for computational convenience. It is well known that
	\begin{equation*}
	\tilde{K}\mathcal{O}(\mathbb{S}^{1})=\{[T\mathbb{S}^{1}],[\mathcal{M}]\},
	\end{equation*}
	where $T\mathbb{S}^{1}$ is the tangent bundle of $\mathbb{S}^{1}$ and $\mathcal{M}$ is the M\"{o}bius bundle. Viewed as groups, it is easily seen that $\tilde{K}\mathcal{O}(\mathbb{S}^{1})\simeq \mathbb{Z}_{2}$, where $T\mathbb{S}^{1}$ is the identity (since it is trivial) and $\mathcal{M}$ is the generator.
	\par
	We begin by computing the index map. Let $U$ be an admissible real Banach space. By Theorem \ref{T10.1.111}, $[\mathbb{S}^{1},\Phi_{0}(U)]\simeq \mathbb{Z}_{2}$. Let $\mathfrak{C}:\mathbb{S}^{1}\to \Phi_{0}(U)$ be the constant map $x\mapsto T$, where $T\in GL(U)$ is fixed.  Since $\mathfrak{Ind}$ is a group homomorphism, the identity must go to the identity and hence, $\mathfrak{Ind}([\mathfrak{C}])=[T\mathbb{S}^{1}]$. Pick a singular operator $T\in\mathcal{S}(U)$ and an open ball $B_\e(T) \subset \Phi_{0}(U)$ of centre $T$ and radius $\varepsilon>0$. Based on the matrix decomposition \eqref{ii.5}, $T$ can be expressed as
	\begin{equation*}
	T=\left(\begin{array}{cc} T_{11} & 0 \\[1ex] 0 & 0 \end{array}\right),
	\end{equation*}
	with $T_{11}\in GL(\Ker[T]^{\perp}, R[T])$. Now, consider the segment $\gamma: J_\e\equiv [-\frac{\varepsilon}{2},\frac{\varepsilon}{2}]\to\Phi_{0}(U)$ defined by
	\begin{equation*}
	\gamma(t):=\left(\begin{array}{cc}
	T_{11} & 0 \\[1ex]
	0 & t I_{n}
	\end{array}\right)=T_{11}\oplus t I_{n},\qquad t\in J_\e,
	\end{equation*}
	where $n=\dim\Ker[T]$ and $I_{n}$ is the identity matrix of rank $n$. Note that
	$\gamma(t)\in GL(U)$ for every $t\in J_\e \setminus\{0\}$, and that $\gamma(J_\e) \subset B_\e(T)$. Now, we re-parameterize this curve in $t$ under an affine transformation to get a curve parameterized in $[0,\frac{1}{2}]$, denoted by $\gamma_{1}:[0,\frac{1}{2}]\to\Phi_{0}(U)$, such that $\gamma_{1}(0), \gamma_{1}(\frac{1}{2})\in GL(U)$. Observe that, since $GL(U)$ is contractible, it is, in particular, path-connected. Thus, there exists a curve $\gamma_{2}:[\frac{1}{2},1]\to\Phi_{0}(U)$ such that $\gamma_{2}(\frac{1}{2})=\gamma_{1}(0)$, $\gamma_{2}(1)=\gamma_{1}(\frac{1}{2})$ and $\gamma_{2}([\frac{1}{2},1])\subset GL(U)$. Consider the curve
	\begin{equation}
	\label{7.7.12}
	\mathfrak{L}:\mathbb{S}^{1}\longrightarrow \Phi_{0}(U), \quad \mathfrak{L}(t):=\left\{\begin{array}{ll}
	\gamma_{1}(t) & \hbox{if}\;\; t\in [0,\frac{1}{2}], \\[1ex]
	\gamma_{2}(t) & \hbox{if}\;\; t\in [\frac{1}{2},1].
	\end{array}\right.
	\end{equation}
	Let $\gamma:[0,1]\to\mathbb{S}^{1}$ be the parametrization of the circle given by $\gamma(t)=(\cos(2\pi t),\sin(2\pi t))$. Then, by the properties of the parity
	\begin{align*}
	\sigma(\mathfrak{L}\circ \gamma,[0,1])&=\sigma(\gamma_{1},[0,\tfrac{1}{2}])\;\sigma(\gamma_{2},[\tfrac{1}{2},1])
	=\sigma(\gamma_{1},[0,\tfrac{1}{2}]) \\[1ex]
	&=\sigma(\gamma,[-\tfrac{\varepsilon}{2},\tfrac{\varepsilon}{2}])= \sigma(T_{11}\oplus t I_{n},[-\tfrac{\varepsilon}{2},\tfrac{\varepsilon}{2}])  \\[1ex]
	&=\sigma(T_{11},[-\tfrac{\varepsilon}{2},\tfrac{\varepsilon}{2}])\;\sigma(t I_{n},[-\tfrac{\varepsilon}{2},\tfrac{\varepsilon}{2}])=\sigma(t I_{n},[-\tfrac{\varepsilon}{2},\tfrac{\varepsilon}{2}]) \\[1ex] &
	=\sign\det (-\tfrac{\varepsilon}{2}  I_{n})\; \sign \det (\tfrac{\varepsilon}{2} I_{n})=-1.
	\end{align*}
	Hence, since $\sigma(\mathfrak{L},\mathbb{S}^{1})=\sigma(\mathfrak{L}\circ \gamma,[0,1])$, we deduce that $\sigma(\mathfrak{L},\mathbb{S}^{1})=-1$. Consequently,
	\begin{equation}
	\label{7.7.13}
	\sigma(\mathfrak{C},\mathbb{S}^{1})=1,\qquad \sigma(\mathfrak{L},\mathbb{S}^{1})=-1.
	\end{equation}
	Thus, since the parity map $\sigma:[\mathbb{S}^{1},\Phi_{0}(U)]\to\mathbb{Z}_{2}$ is an isomorphism, it becomes apparent that $[\mathfrak{C}]\neq [\mathfrak{L}]$ on $[\mathbb{S}^{1},\Phi_{0}(U)]$. Therefore, the index map is necessarily given by
	\begin{equation*}
	\mathfrak{Ind}: [\mathbb{S}^{1},\Phi_{0}(U)] \longrightarrow \tilde{K}\mathcal{O}(\mathbb{S}^{1})
	\qquad \left\{\begin{array}{l}  [\mathfrak{C}]\mapsto [T\mathbb{S}^{1}], \\[1ex] [\mathfrak{L}]\mapsto [\mathcal{M}].\end{array}\right.
	\end{equation*}
	To compute $\L$, we still have to calculate $\mathfrak{I}[E]:\pi_{1}(\mathbb{S}^{1})\to \mathbb{Z}_{2}$ for every $E\in \tilde{K}\mathcal{O}(\mathbb{S}^{1})$.
	\par
	Suppose $E=[T\mathbb{S}^{1}]$. Then, regarding $\mathbb{S}^{1}$ as the unit circle in $\mathbb{C}$, $|z|=1$, and setting
	\begin{equation*}
	\pi_{1}(\mathbb{S}^{1})=\{[\gamma_{n}]: \gamma_{n}:\mathbb{S}^{1}\to\mathbb{S}^{1}, \gamma_{n}(z)=z^{n}\}\simeq \mathbb{Z},
	\end{equation*}
	it follows that, for every $n\in\mathbb{Z}$,
	\begin{align*}
	\mathfrak{I}[E]([\gamma_{n}])=
	\chi_{2}[\mathfrak{Ind}^{-1}[E]\circ\gamma_{n},\mathbb{S}^{1}]=
	\chi_{2}[\mathfrak{C}\circ\gamma_{n},\mathbb{S}^{1}]=\chi_{2}[T,\mathbb{S}^{1}]=1,
	\end{align*}
	by the properties of $\chi$ discussed in Section 2. Thus, $\mathfrak{I}[E]\equiv 1$ and therefore
	\begin{equation*}
	\L([T\mathbb{S}^{1}])=\int_{\mathcal{L}_{\mathbf{x}}(\mathbb{S}^{1})}\mathfrak{I}[E]([\gamma])\ d\mu_{\mathbf{x}}(\gamma)=\int_{\mathcal{L}_{\mathbf{x}}(\mathbb{S}^{1})}1\cdot d\mu_{\mathbf{x}}(\gamma)=1.
	\end{equation*}
	Note that this value can be also found, directly, by applying Theorem \ref{T7.2.2}, as the trivial bundle is orientable.
	\par
	Subsequently, we suppose that $E=[\mathcal{M}]$. Then, since $\gamma_{n}=\gamma_{1}\circ \overset{n}{\cdots} \circ \gamma_{1}$, by the product formula of the multiplicity, it becomes apparent that, for every $n\in\mathbb{Z}$,
	\begin{align*}
	\mathfrak{I}[E]([\gamma_{n}])&=
	\chi_{2}[\mathfrak{Ind}^{-1}[E]\circ\gamma_{n},\mathbb{S}^{1}] =\prod_{i=1}^{n}\chi_{2}[\mathfrak{Ind}^{-1}[E]\circ\gamma_{1},\mathbb{S}^{1}]\\
	&=\prod_{i=1}^{n}\chi_{2}[\mathfrak{L}\circ\gamma_{1},\mathbb{S}^{1}]=
	\prod_{i=1}^{n}\chi_{2}[\mathfrak{L},\mathbb{S}^{1}].
	\end{align*}
	Next, we will determine $\chi_{2}[\mathfrak{L},\mathbb{S}^{1}]$ through two different techniques. The first one is heuristic and uses the relationship between the algebraic multiplicity and the intersection index. By the definition of $\mathfrak{L}$,
	the only intersection of $\mathfrak{L}$ with the singular variety $\mathcal{S}(U)$ occurs at the point $T$ and this intersection is transversal, as illustrated by Figure \ref{F7.1}. Thus,   $i_{2}(\mathfrak{L},\mathbb{S}^{1})=-1$ and therefore
	\begin{equation*}
	\chi_{2}[\mathfrak{L},\mathbb{S}^{1}]=i_{2}(\mathfrak{L},\mathbb{S}^{1})=-1.
	\end{equation*}
	The second technique is rigorous and uses the relationship of the multiplicity $\chi$ with the parity materialized by Theorem \ref{th4.8}. Indeed, by \eqref{7.7.13}, we have that
	\begin{equation*}
	\chi_{2}[\mathfrak{L},\mathbb{S}^{1}]=\sigma(\mathfrak{L},\mathbb{S}^{1})=-1.
	\end{equation*}
	Thus, for every $n\in\mathbb{Z}$,
	\begin{align*}
	\mathfrak{I}[E]([\gamma_{n}])=\prod_{i=1}^{n}\chi_{2}[\mathfrak{L},\mathbb{S}^{1}]=(-1)^{n}.
	\end{align*}
	Hence, we find that
	\begin{align*}
	\L(E) =\int_{\mathcal{L}_{\mathbf{x}}(\mathbb{S}^{1})}\mathfrak{I}[E]([\gamma]) \ d\mu_{\mathbf{x}}(\gamma)
	= \sum_{n\in\mathbb{Z}}(-1)^{n}\cdot \mu_{\mathbf{x}}([\gamma_{n}]).
	\end{align*}
	
	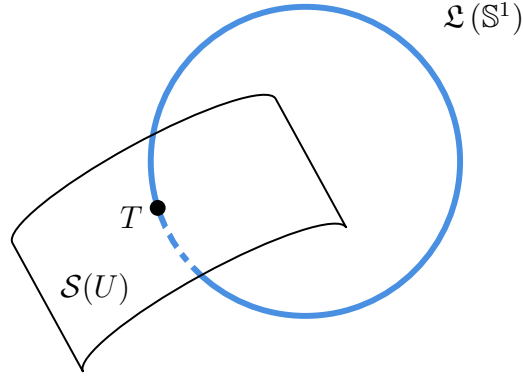
\begin{figure}
		\begin{center}

			\tikzset{every picture/.style={line width=0.75pt}} 
			
			\begin{tikzpicture}[x=0.75pt,y=0.75pt,yscale=-1,xscale=1]
			
			\draw  [color={rgb, 255:red, 74; green, 144; blue, 226 }  ,draw opacity=1 ][line width=2.25]  (79,90) .. controls (79,46.92) and (113.92,12) .. (157,12) .. controls (200.08,12) and (235,46.92) .. (235,90) .. controls (235,133.08) and (200.08,168) .. (157,168) .. controls (113.92,168) and (79,133.08) .. (79,90) -- cycle ;
			\draw   (142.01,58.46) -- (177.64,123.49) .. controls (173.89,116.65) and (141.13,127.38) .. (104.47,147.47) .. controls (67.81,167.56) and (41.13,189.39) .. (44.88,196.23) -- (9.25,131.2) .. controls (5.5,124.36) and (32.18,102.53) .. (68.84,82.44) .. controls (105.5,62.36) and (138.26,51.62) .. (142.01,58.46) -- cycle ;
			\draw  [fill={rgb, 255:red, 0; green, 0; blue, 0 }  ,fill opacity=1 ] (79,113.5) .. controls (79,111.57) and (80.57,110) .. (82.5,110) .. controls (84.43,110) and (86,111.57) .. (86,113.5) .. controls (86,115.43) and (84.43,117) .. (82.5,117) .. controls (80.57,117) and (79,115.43) .. (79,113.5) -- cycle ;
			\draw  [color={rgb, 255:red, 255; green, 255; blue, 255 }  ,draw opacity=1 ][line width=3] [line join = round][line cap = round] (82.33,129.33) .. controls (84.69,126.97) and (92.33,121.77) .. (92.33,120.33) ;
			\draw  [color={rgb, 255:red, 255; green, 255; blue, 255 }  ,draw opacity=1 ][line width=3] [line join = round][line cap = round] (87.33,138.33) .. controls (91.68,138.33) and (96.54,130.33) .. (101.33,130.33) ;
			\draw  [color={rgb, 255:red, 255; green, 255; blue, 255 }  ,draw opacity=1 ][line width=3] [line join = round][line cap = round] (91.33,148.33) .. controls (93.2,146.47) and (104.33,137.17) .. (104.33,136.33) ;
			
			\draw (225,8) node [anchor=north west][inner sep=0.75pt]    {$\mathfrak{L}\left(\mathbb{S}^{1}\right)$};
			\draw (32,145) node [anchor=north west][inner sep=0.75pt]    {$\mathcal{S}( U)$};
			\draw (62,111) node [anchor=north west][inner sep=0.75pt]    {$T$};

			\end{tikzpicture}
		\end{center}
		\caption{Transversal intersection of $\mathfrak{L}$ with $\mathcal{S}(U)$.}
		\label{F7.1}
	\end{figure}

	\par
	Now, we proceed to the computation of $\mu_{\mathbf{x}}([\gamma_{n}])$ through \eqref{7.7.9}. Let us consider the circle $\mathbb{S}^{1}$ as the quotient $\mathbb{R}/ 2\sqrt{\pi}\mathbb{Z}$. It is well known that the universal covering of $M=\mathbb{S}^{1}$ is $\tilde{M}=\mathbb{R}$ with corresponding covering map
	\begin{equation*}
	\pi:\tilde{M}\longrightarrow M, \quad x\mapsto [x],
	\end{equation*}
	where $[x]$ denotes the class of $x\in\mathbb{R}$ in the quotient $\mathbb{R}/ 2\sqrt{\pi}\mathbb{Z}$. It is easily seen that
	$$
	\Aut_{M}\tilde{M}=\{\varphi^{n} : n\in \mathbb{Z}\},
	$$
	where $\varphi^{n}(x)=x+2\sqrt{\pi}n$, $x\in\mathbb{R}$, for each $n\in\mathbb{N}$. Since the heat kernel of the universal covering space $\tilde{M}=\mathbb{R}$ is
	\begin{equation*}
	\tilde{p}_{t}(x,y)=\frac{1}{\sqrt{4\pi  t}}\exp\left\{-\frac{|x-y|^{2}}{4t} \right\},
	\end{equation*}
	it follows from \eqref{7.7.9} that, for every $n\in\mathbb{Z}$,
	\begin{align*}
	\mu_{\mathbf{x}}(\mathcal{L}^{\varphi^{n}}_{\mathbf{x}}(\mathbb{S}^{1}))&=\frac{\tilde{p}_{1}(\tilde{\mathbf{x}},\varphi^{n}(\tilde{\mathbf{x}}))}{\sum_{m\in\mathbb{Z}}\tilde{p}_{1}(\tilde{\mathbf{x}},\varphi^{m}(\tilde{\mathbf{x}}))}\\
	&=\frac{\exp\{-\pi n^{2}\}}{\sum_{m\in\mathbb{Z}}\exp\{-\pi m^{2}\}}=\frac{\Gamma\left(\frac{3}{4}\right)}{\sqrt[4]{\pi}}\exp\{-\pi n^{2}\}.
	\end{align*}
	For the last identity, we have used that
	\begin{equation}
	\label{7.7.14}
	\sum_{m=-\infty}^{\infty}e^{-\pi m^{2}}=\theta_{3}(0,e^{-\pi})=\frac{\sqrt[4]{\pi}}{\Gamma\left(\frac{3}{4}\right)},
	\end{equation}
	where $\theta_{3}(z,q)$ stands for  the Jacobi  Theta function (see \cite{Yi}, if necessary). Hence,
	\begin{equation*}
	\L(E)=\int_{\mathcal{L}_{\mathbf{x}}(\mathbb{S}^{1})}\mathfrak{I}[E]([\gamma]) \ d\mu_{\mathbf{x}}(\gamma)=\frac{\Gamma\left(\frac{3}{4}\right)}{\sqrt[4]{\pi}} \sum_{n\in\mathbb{Z}}(-1)^{n} \exp(-\pi n^{2}).
	\end{equation*}
	Again, a simple computation with the Jacobi Theta function yields to
	\begin{equation}
	\label{7.7.15}
	\sum_{n\in\mathbb{Z}}(-1)^{n} \exp(-\pi n^{2})=\sqrt[4]{\frac{\pi}{2}}\frac{1}{\Gamma\left(\frac{3}{4}\right)},
	\end{equation}
	which implies that
	\begin{equation*}
	\L([\mathcal{M}])=\frac{1}{\sqrt[4]{2}}.
	\end{equation*}
	In particular, by Theorem \ref{T7.2.2}, since $\L([\mathcal{M}])\neq 1$,  it follows that $\mathcal{M}$ is not orientable. So, our analysis establishes a new (different) proof  of this well known result.
	\par
	Therefore, we have proved that the global torsion invariant $\L$ of the circle is given by
	\begin{equation*}
	\L:\tilde{K}\mathcal{O}(\mathbb{S}^{1})\longrightarrow [-1,1], \qquad \L([T\mathbb{S}^{1}])=1, \quad \L([\mathcal{M}])=\frac{1}{\sqrt[4]{2}}.
	\end{equation*}
	
	As a direct application of the additive formula of Proposition \ref{P7.2.4}, we can obtain $\L([T\mathbb{S}^{1}])$ from $[\mathcal{M}]$. Since
	$$
	[\mathcal{M}]\oplus[\mathcal{M}]=[T\mathbb{S}^{1}]
	$$
	and
	$$
	\chi_{2}[\mathfrak{Ind}^{-1}[\mathcal{M}]\circ\gamma_{n},\mathbb{S}^{1}]=(-1)^{n}\qquad \hbox{for all} \;\; n\in\mathbb{Z},
	$$
	from Proposition \ref{P7.2.4} it is apparent that
	\begin{align*}
	\L([T\mathbb{S}^{1}]) & =\sum_{n\in\mathbb{Z}}\chi_{2}[\mathfrak{Ind}^{-1}[\mathcal{M}]\circ\gamma_{n},\mathbb{S}^{1}]\cdot\chi_{2}[\mathfrak{Ind}^{-1}[\mathcal{M}]\circ\gamma_{n},\mathbb{S}^{1}]\cdot\mu_{\mathbf{x}}([\gamma_{n}])\\
	&=\sum_{n\in\mathbb{Z}}(-1)^{n}(-1)^{n}\mu_{\mathbf{x}}([\gamma_{n}])=\sum_{n\in\mathbb{Z}}\mu_{\mathbf{x}}([\gamma_{n}])=1.
	\end{align*}
	
	\subsection{Global torsion invariant of $\mathbb{T}^{2}$} In this subsection, we will compute the global torsion invariant $\L(\tilde{K}\mathcal{O}(\mathbb{T}^{2}))$ of the torus considering it as the quotient
	\begin{equation}
	\label{7.7.16}
	\mathbb{T}^{2}:=\mathbb{S}^{1}\times\mathbb{S}^{1} \equiv\mathbb{R}^{2}/[2\sqrt{\pi}\mathbb{Z}\times 2\sqrt{\pi}\mathbb{Z}].
	\end{equation}
	It is well known that the universal covering of $M=\mathbb{T}^{2}$ is $\tilde{M}=\mathbb{R}^{2}$ with corresponding covering map
	\begin{equation*}
	\pi:\tilde{M}\longrightarrow M, \quad (x,y)\mapsto [(x,y)],
	\end{equation*}
	where $[(x,y)]$ denotes the class of $(x,y)\in\mathbb{R}^{2}$ in the quotient \eqref{7.7.16}. An easy computation shows that
	\begin{equation*}
	\Aut_{M}\tilde{M}=\{\varphi^{n_{1},n_{2}}: n_{1},n_{2}\in\mathbb{Z}\},
	\end{equation*}
	where the morphisms $\varphi^{n_{1},n_{2}}:\mathbb{R}^{2}\to\mathbb{R}^{2}$ are defined, for every
	$n_{1},n_{2}\in\mathbb{Z}$,  by
	$$
	\varphi^{n_{1},n_{2}}(x,y)=(x+2\sqrt{\pi}n_{1},y+2\sqrt{\pi}n_{2}).
	$$
	The group of isomorphisms is given by $\pi_{1}(\mathbb{T}^{2})\simeq \Aut_{M}\tilde{M}\simeq \mathbb{Z}\oplus \mathbb{Z}$. To compute the values of the global torsion invariant of $\mathbb{T}^{2}$, $\L(\tilde{K}\mathcal{O}(\mathbb{T}^{2}))$, we will use \eqref{7.7.11}. According to it,
	\begin{equation}
	\label{7.7.17}
	\L(\tilde{K}\mathcal{O}(M))=\left\{\frac{\sum_{\varphi}\zeta(\varphi)\cdot \tilde{p}_{1}(\tilde{\mathbf{x}},\varphi(\tilde{\mathbf{x}}))}{\sum_{\varphi}\tilde{p}_{1}(\tilde{\mathbf{x}},\varphi(\tilde{\mathbf{x}}))} : \zeta\in \Hom(\Aut_{\tilde{M}}M,\mathbb{Z}_{2}) \right\}.
	\end{equation}
	Since the heat kernel of the universal covering space $\tilde{M}=\mathbb{R}^{2}$ is
	\begin{equation*}
	\tilde{p}_{t}(x,y)=\frac{1}{4\pi  t}\exp\left\{-\frac{|x-y|^{2}}{4t} \right\}, \quad (x,y)\in\mathbb{R}^{2}, \ t>0,
	\end{equation*}
	it follows that, for every $\tilde{\mathbf{x}}\in \pi^{-1}(\mathbf{x})$,
	\begin{equation*}
	\tilde{p}_{1}(\tilde{\mathbf{x}},\varphi^{n_{1},n_{2}}(\tilde{\mathbf{x}}))=\frac{1}{4\pi }\exp\left\{ -\pi (n_{1}^{2}+n_{2}^{2})\right\}.
	\end{equation*}
	Hence, using the summability methods involving theta functions \eqref{7.7.14} we obtain that
	\begin{align*}
	\sum_{\varphi\in \Aut_{\tilde{M}}M} \tilde{p}_{1}(\tilde{\mathbf{x}},\varphi(\tilde{\mathbf{x}}))&=\sum_{n_{1},n_{2}\in\mathbb{Z}}\tilde{p}_{1}(\tilde{\mathbf{x}},\varphi^{n_{1},n_{2}}(\tilde{\mathbf{x}}))\\
	&=\frac{1}{4\pi}\left(\sum_{n\in\mathbb{Z}}e^{-\pi n^{2}}\right)^{2}=\frac{1}{4\pi}\frac{\sqrt{\pi}}{\Gamma^{2}\left(\frac{3}{4}\right)}.
	\end{align*}
	We proceed to compute the group $\Hom(\Aut_{\tilde{M}}M,\mathbb{Z}_{2})$. Clearly, since $\Aut_{\tilde{M}}M$ is generated by the transformations $\varphi^{1,0}$ and  $\varphi^{0,1}$, every homomorphism $\zeta:\Aut_{\tilde{M}}M\to\mathbb{Z}_{2}$ is determinated by the values $\zeta(\varphi^{1,0}), \zeta(\varphi^{0,1})\in\mathbb{Z}_{2}$. In this way, we obtain the group isomorphisms
	\begin{equation*}
	\Hom(\Aut_{\tilde{M}}M,\mathbb{Z}_{2})\simeq \Hom(\mathbb{Z}\oplus\mathbb{Z},\mathbb{Z}_{2})\simeq \mathbb{Z}_{2}\oplus\mathbb{Z}_{2}.
	\end{equation*}
	As a direct consequence, given any homomorphism $\zeta\in \Hom(\Aut_{\tilde{M}}M,\mathbb{Z}_{2})$, we can write the action of $\zeta$ on each $\varphi^{n_{1},n_{2}}\in \Aut_{\tilde{M}} M$ as
	$$
	\zeta(\varphi^{n_{1},n_{2}})=[\zeta(\varphi^{1,0})]^{n_{1}}\cdot[\zeta(\varphi^{0,1})]^{n_{2}}.
	$$
	This allows us to compute the sum
	\begin{align*}
	\sum_{\varphi\in\Aut_{\tilde{M}}M}\zeta(\varphi)\cdot \tilde{p}_{1}(\tilde{\mathbf{x}},\varphi(\tilde{\mathbf{x}}))&=\sum_{n_{1},n_{2}\in\mathbb{Z}}[\zeta(\varphi^{1,0})]^{n_{1}}\cdot[\zeta(\varphi^{0,1})]^{n_{2}}\cdot\tilde{p}_{1}(\tilde{\mathbf{x}},\varphi^{n_{1},n_{2}}(\tilde{\mathbf{x}}))\\
	&=\frac{1}{4\pi}\left(\sum_{n\in\mathbb{Z}}[\zeta(\varphi^{1,0})]^{n}e^{-\pi n^{2}}\right)\left(\sum_{n\in\mathbb{Z}}[\zeta(\varphi^{0,1})]^{n}e^{-\pi n^{2}}\right),
	\end{align*}
	where $\zeta(\varphi^{1,0}),\zeta(\varphi^{0,1})\in\mathbb{Z}_{2}$ depend on the chosen $\zeta\in\Hom(\Aut_{\tilde{M}}M,\mathbb{Z}_{2})$. Hence, by \eqref{7.7.17}, we can deduce that
	\begin{align*}
	\L(\tilde{K}\mathcal{O}(\mathbb{T}^{2}))&
	=\left\{\frac{\Gamma^{2}\left(\frac{3}{4}\right)}{\sqrt{\pi}} \left(\sum_{n\in\mathbb{Z}}\alpha^{n}e^{-\pi n^{2}}\right)\left(\sum_{n\in\mathbb{Z}}\beta^{n}e^{-\pi n^{2}}\right) : \alpha,\beta\in\mathbb{Z}_{2} \right\}\\
	&=\Big\{1,\frac{1}{\sqrt[4]{2}},\frac{1}{\sqrt[4]{2}},\frac{1}{\sqrt{2}}\Big\},
	\end{align*}
	where \eqref{7.7.14} and \eqref{7.7.15} have been used in the last step. This information has been represented in the table of Figure \ref{F7.2}. According to Theorem \ref{th5.5}, $\mathfrak{I}$ defines an isomorphism between $\Vect_{1}(\mathbb{T}^{2})$ and $\Hom(\pi_{1}(\mathbb{T}^{2}),\mathbb{Z}_{2})\simeq \Hom(\Aut_{\tilde{M}}M,\mathbb{Z}_{2})$. Hence $\Vect_{1}(\mathbb{T}^{2})\simeq \Hom(\Aut_{\tilde{M}}M,\mathbb{Z}_{2})$ and therefore, each $\zeta\in \Hom(\Aut_{\tilde{M}}M,\mathbb{Z}_{2})$ corresponds to a single isomorphism class of line bundle. Each row of the table corresponds to an isomorphism class of line bundle. So,
	the  table describes the values of the global torsion invariant on each line bundle.
	\begin{figure}[h!]
		\begin{tabular}{ |r|r|l|l| }
			\hline
			$\zeta(\varphi^{1,0})$ & $\zeta(\varphi^{0,1})$ & $\L$ \\
			\hline
			$1$ & $1$ & $1$ \\
			$1$ & $-1$ & $1/\sqrt[4]{2}$ \\
			$-1$ & $1$ & $1/\sqrt[4]{2}$ \\
			$-1$ & $-1$ & $1/\sqrt{2}$ \\
			\hline
		\end{tabular} \quad \quad \begin{tabular}{ |r|r|r|l| }
			\hline
			$\zeta(\varphi^{1,0,0})$ & $\zeta(\varphi^{0,1,0})$ & $\zeta(\varphi^{0,0,1})$ & $\L$ \\
			\hline
			$1$ & $1$ & $1$ & $1$ \\
			$-1$ & $1$ & $1$ & $1/\sqrt[4]{2}$ \\
			$1$ & $-1$ & $1$ & $1/\sqrt[4]{2}$ \\
			$1$ & $1$ & $-1$ & $1/\sqrt[4]{2}$ \\
			$-1$ & $1$ & $-1$ & $1/\sqrt{2}$ \\
			$1$ & $-1$ & $-1$ & $1/\sqrt{2}$ \\
			$-1$ & $-1$ & $1$ & $1/\sqrt{2}$ \\
			$-1$ & $-1$ & $-1$ & $1/\sqrt[4]{8}$ \\
			\hline
		\end{tabular}
		\caption{The global torision invariant for $\mathbb{T}^{2}$ and $\mathbb{T}^{3}$, respectively.}
		\label{F7.2}
	\end{figure}
	\par
	Rephrasing the computations, we can obtain the corresponding result for the $n$-dimensional torus $\mathbb{T}^{n}:=\bigtimes_{i=1}^{n}\mathbb{S}^{1}$, where each factor is taken as $\mathbb{S}^{1}\equiv \mathbb{R}/2\sqrt{\pi}\mathbb{Z}$. In this case, the values of the global torsion invariant are given by
	\begin{equation*}
	\L(\tilde{K}\mathcal{O}(\mathbb{T}^{n}))=\Big\{\left(\frac{1}{\sqrt[4]{2}}\right)^{m}:m\in\{1,2,\cdots,n\}\Big\}.
	\end{equation*}
	It is also possible to compute the corresponding tables for the $n$-dimensional torus. The corresponding table for $n=3$ is given in the second table of Figure \ref{F7.2}. This table gives the values of the global torsion invariant on each line bundle of $\mathbb{T}^{3}$.

\appendix

\section{Wiener Measure on Loops Spaces}

\noindent For any given Riemannian manifold, $(M,g)$, let denote by $\mathcal{B}_{M}$ the Borel $\sigma$-algebra of $M$, and consider, in $(M,g)$, the measure $\mu: \mathcal{B}_{M}\to [0,+\infty]$ induced by the metric $g$. Locally, this measure can be expressed by
\begin{equation*}
d\mu=\sqrt{\det(g_{ij})_{ij}} \ dx_{1}\wedge \cdots \wedge dx_{m}
\end{equation*}
where $m$ is the dimension of $M$ and $(g_{ij})_{ij}$ is the matrix of $g$ in a local chart. According to B\"{a}r and Pf\"{a}ffle \cite{B3} and  Grigor'yan  \cite{1.7}, for any given closed Riemannian manifold, $(M,g)$, there exists a heat kernel $p_{t}(x,y)$, $t>0$, $x,y\in M$. Namely, the Schwartz kernel of the self-adjoint semigroup $e^{t\Delta}$ on $L^{2}(M,\mu)$, where $\Delta$ stands for  the Laplace--Beltrami operator on $(M,g)$.
\par
For any given (fixed) $\mathbf{x}\in M$, the Wiener measure on the loop space
$$
   \mathcal{L}_{\mathbf{x}}(M):=\{\gamma\in\mathcal{C}([0,1],M):\gamma(0)=\gamma(1)=\mathbf{x}\}
$$
is a measure $\lambda_{\mathbf{x}}:\mathcal{B}_{\mathbf{x}}\to[0,+\infty]$ on the measurable space $(\mathcal{L}_{\mathbf{x}}(M),\mathcal{B}_{\mathbf{x}})$, where $\mathcal{B}_{\mathbf{x}}$ stands for the Borel $\sigma$-algebra of $\mathcal{L}_{\mathbf{x}}(M)$ with respect to the topology of the uniform convergence, such that, for every finite subset
$$
 \mathcal{T}=\{0=t_{0}<t_{1}<\cdots<t_{n}<t_{n+1}=1\}\subset[0,1]
$$
and any $(B_{t})_{t\in\mathcal{T}\backslash\{0,1\}}\subset \mathcal{B}_{M}$,
\begin{equation}\label{E4}
\lambda_{\mathbf{x}}\left(\pi^{-1}_{\mathcal{T}}(B_{t})_{t\in\mathcal{T}\backslash\{0,1\}}\right)
=\int_{B_{t_{1}}}\overset{(n)}{\cdots}\int_{B_{t_{n}}} \prod_{i=1}^{n+1}p_{t_{i}-t_{i-1}}(x_{i},x_{i-1})\prod_{i=1}^{n}d\mu(x_{i}), \quad x_{0}=x_{n+1}=\mathbf{x}.
\end{equation}
In this context, we are using using the notation $\pi_{\mathcal{T}}$ to denote the projector
\begin{equation*}
\pi_{\mathcal{T}}:M^{[0,1]}\longrightarrow M^{\mathcal{T}\backslash\{0,1\}}, \qquad \pi_{\mathcal{T}}(\gamma_{t})_{t\in[0,1]}:=(\gamma_{t})_{t\in \mathcal{T}\backslash\{0,1\}}.
\end{equation*}
Since
\begin{equation}\label{J}
\lambda_{\mathbf{x}}(\mathcal{L}_{\mathbf{x}}(M))=p_{1}(\mathbf{x},\mathbf{x})>0,
\end{equation}
the measure $\lambda_{\mathbf{x}}$ is not a probability measure, unless $p_1(\mathbf{x},\mathbf{x})=1$. Nevertheless, the normalized measure
$$
  \mu_{\mathbf{x}}=p_{1}(\mathbf{x},\mathbf{x})^{-1}\lambda_{\mathbf{x}}
$$
provides us with a probability measure. Rephrasing \eqref{E4}, it is apparent that, for every finite subset $$
  \mathcal{T}=\{0=t_{0}<t_{1}<\cdots<t_{n}<t_{n+1}=1\}\subset[0,1]
$$
and each $(B_{t})_{t\in\mathcal{T}\backslash\{0,1\}}\subset \mathcal{B}_{M}$,
\begin{equation*}
\mu_{\mathbf{x}}(\pi^{-1}_{\mathcal{T}}(B_{t})_{t\in\mathcal{T}\backslash\{0,1\}})=\int_{B_{t_{1}}}\overset{(n)}{\cdots}\int_{B_{t_{n}}} p_{1}(\mathbf{x},\mathbf{y})^{-1} \prod_{i=1}^{n+1}p_{t_{i}-t_{i-1}}(x_{i},x_{i-1})\prod_{i=1}^{n}d\mu(x_{i}), \quad x_{0}=x_{n+1}=\mathbf{x}.
\end{equation*}
The measure $\mu_{\mathbf{x}}:\mathcal{B}_{\mathbf{x}}\to[0,+\infty]$ is ususally refereed to as  the \textit{normalized} Wiener measure.
\par
This construction can be easily generalized to cover pinned spaces
$$
   \mathcal{C}^{\mathbf{y}}_{\mathbf{x}}(M):=\{\gamma\in\mathcal{C}([0,1],M):\gamma(0)=\mathbf{x}, \gamma(1)=\mathbf{y}\},
$$
where it is possible to construct a generalized Wiener measure,  $\lambda_{\mathbf{x}}^{\mathbf{y}}:\mathcal{B}_{\mathbf{x}}^{\mathbf{y}}\to[0,+\infty]$, as well as
a normalized Wiener measure, $\mu_{\mathbf{x}}^{\mathbf{y}}:\mathcal{B}_{\mathbf{x}}^{\mathbf{y}}\to[0,+\infty]$, where $\mathcal{B}_{\mathbf{x}}^{\mathbf{y}}$ stands for the Borel $\sigma$-algebra of $\mathcal{C}^{\mathbf{y}}_{\mathbf{x}}(M)$ under the topology of the uniform convergence (see B\"{a}r
and Pf\"{a}ffle \cite{B3} for any further required detail).

\end{document}